\documentclass[11pt,reqno]{amsart}
\usepackage{amssymb}
\usepackage[T1]{fontenc}
\usepackage[dvipsnames]{xcolor}
\usepackage{palatino}
\input amssym.def
\usepackage{amsmath, amsfonts}
\usepackage{enumitem}
\usepackage{amscd, mathtools}
\usepackage[mathscr]{eucal}

\newfont{\cyrr}{wncyr10}

\newcommand{\thmref}[1]{Theorem~\ref{#1}}
\newtheorem{thm}{Theorem}

\newtheorem{lem}[thm]{Lemma}

\newtheorem{prop}[thm]{Proposition}

\newtheorem{rmk}{Remark}[section]

\newcommand{\propref}[1]{Proposition~\ref{#1}}

\newcommand{\lemref}[1]{Lemma~\ref{#1}}

\newcommand{\C}{{\mathbb C}}

\newcommand{\DD}{\sideset{}{'}\sum_{X < d \leq 2X \atop{(d, 2q)=1}}}
\newcommand{\DDD}{\sum_{\substack{X < d \le 2X \\ d \in T(\psi_0, q)} }}
\newcommand{\bb}{\frac{X}{a^2} < b \leq \frac{2X}{a^2}}
\newcommand{\pp}{\prod_{L^2 \leq p \leq L^4}}
\newcommand{\cchi}{\chi_{8d}}
\newcommand{\poc}{\psi \otimes \chi_{8d}}
\newcommand{\dsi}{d_{\psi}}
\newcommand{\sgn}{\textrm{sgn}}

\begin{document}
\title[Extreme Values]{On extreme values of quadratic twists 
of Dirichlet-type $L$-functions}

\author{Sanoli Gun and Rashi Lunia}

\address{Sanoli Gun \\ \newline
The Institute of Mathematical Sciences, A CI of Homi Bhabha National Institute, 
CIT Campus, Taramani, Chennai 600 113, India.}
\email{sanoli@imsc.res.in}
	
\address{Rashi Lunia \\ \newline
Alfr\'ed R\'enyi Institute of Mathematics, Re\'altanoda Street 13--15, H-1364, Budapest, Hungary.}
\email{lunia.rashi@renyi.hu}

\begin{abstract}
In a recent work \cite{GKS}, it has been shown that $L$-functions
associated with arbitrary non-zero cusp forms take large
values at the central critical point.
The goal of this note is to derive analogous results for 
twists of Dirichlet-type functions. More precisely, 
for an odd integer $q >1$, let $F$ be a non-zero 
$\C$-linear combination of complex Dirichlet characters $\psi$
modulo $q$ such that $\psi^2$ is primitive modulo $q$.
We show that for any $\epsilon>0$ and sufficiently large $X$,
there are $\gg X^{1-\epsilon}$ fundamental discriminants $8d$ 
with $X < |d| \leq 2X$ and ${(d, 2q)=1}$ such that 
${|L(1/2, F \otimes \cchi)| }$ is large. 
\end{abstract}	

\maketitle

\section{Introduction}

For a complex number $s$ with $\Re(s) >1$ and a Dirichlet character $\chi$ modulo $q$,
the $L$-function associated with it is defined by
$$
L(s, \chi) = \sum_{n=1}^{\infty}\frac{\chi(n)}{n^s}.
$$
This $L$-function has a meromorphic continuation to the entire complex plane and satisfies a 
functional equation relating its value at $s$ with its value at $1-s$.
Consequently, values of these $L$-functions at the central 
point $s=1/2$ are interesting as well as mysterious.

Jutila \cite{MJ} obtained asymptotic formulae for 
first and second moments of Dirichlet $L$-functions attached to 
real characters. Application of Cauchy-Schwarz inequality ensures 
that there are $\gg \frac{X}{\log X}$ real characters $\chi$ of 
conductor at most $X$ for which $L(1/2, \chi) \neq 0$.
This result was improved by Soundararajan in \cite{KS00} 
to show that $L(1/2, \chi) \neq 0$ for at least $87.5 \%$ of the 
real characters $\chi$.

In his thesis, Soundararajan \cite{KST} showed that for an even, non-quadratic,
primitive Dirichlet character $\psi$ of odd conductor $q$, 
$L(1/2, \psi \otimes \cchi) \neq 0$ for at least $20\%$
of the square-free integers $d \geq 0$ with $(d, 2q)=1$ and indicated
that it can be improved to $33\%$.
In a novel article \cite{KS}, Soundararajan introduced the so-called
resonance method 
to study large values of Dirichlet $L$-functions at $1/2$.
This method was used by the first author along with Kohnen
and Soundararajan \cite{GKS} to obtain large values of
$L$-functions associated with arbitrary non-zero cusp forms
at central critical point. The goal of this article is to extend
this investigation to the study of large values of Dirichlet-type $L$-functions.
More precisely, we prove the following theorem.
 
\begin{thm}\label{mainthm}
Let $q > 1$ be an odd integer and 
$$
F= \sideset{}{^*}\sum_{\psi ({\rm mod} q)} c_{\psi} \psi
$$
be a non-zero linear combination of non-quadratic Dirichlet 
characters $\psi$ modulo $q$ such that $\psi^2$ is primitive modulo $q$.
Also let $c_\psi \in \C$.
For any sufficiently small real number $ \epsilon > 0$ and sufficiently large real number $X$, there exists 
$\gg X^{1- \epsilon}$ fundamental discriminants $8d$ 
with $X < |d| \leq 2X$ and $(d, 2q)=1$ such that 
$$
|L(1/2, F \otimes \cchi)|
~>~
\exp\left(\frac{1}{80}\sqrt{\frac{\log X}{\log\log X}}\right).
$$
\end{thm}

\begin{rmk}\label{mainrmk1}
Note that it is sufficient to prove \thmref{mainthm} when $F$ is a non-zero
$\C$-linear combination of even, non-quadratic Dirichlet characters $\psi$ of odd 
conductor $q >1$ such that $\psi^2$ is primitive modulo~$q$. Indeed, we can 
write 
$$
F ~=~ F_{\rm even} ~+~ F_{\rm odd}~, 
$$
where $F_{\rm even}$ (resp. $F_{\rm odd}$) is a $\C$-linear combination of 
even (resp. odd) characters $\psi$. When $d>0$ (resp. $d <0$), we see that
$L(1/2, F_{\rm odd} \otimes \chi_{8d}) = 0$ 
(resp. $L(1/2, F_{\rm even} \otimes \chi_{8d}) = 0$).

\end{rmk}

This article is arranged as follows. In \S2, we recall some preliminaries and
state the results leading to the proof of the main theorem.  In \S3, \S4 and \S5, we prove 
the auxiliary results. Finally in \S6, we prove \thmref{mainthm}.

\smallskip

\section{Preliminaries}
Throughout the article, $p, \ell$ denote prime numbers and for $z \in \mathbb{C}$,
$e(z)$ denotes $e^{2\pi i z}$. Henceforth $\epsilon > 0$ is a sufficiently small
real number which may vary from line to line.
For an integer $D \equiv 0, 1 \!\! \pmod{4}$, let $\chi_D(\cdot)=\Big(\frac{D}{\cdot}\Big)$ 
be a real Dirichlet character with conductor less than equal to $|D|$.
Let $q >1$ be an odd integer and $\psi$ be a non-quadratic, 
primitive Dirichlet character modulo $q$. 
For simplicity, we shall assume $\psi(-1)=1$. For a positive squarefree
integer $d$, if $(d, 2q)=1$, the character $\psi \otimes \cchi$ defined as 
${\psi \otimes \cchi(\cdot)=\psi(\cdot)\cchi(\cdot)}$,  
is an even primitive character with conductor $8dq$.
The Dirichlet $L$-function attached to $\psi \otimes \cchi$
is defined as 
$$
L(s, \psi \otimes \cchi)
=\sum_{n=1}^{\infty}\frac{\psi(n)\cchi(n)}{n^s},
\qquad \Re(s)>1.
$$
Since $\psi$ is non-quadratic, $\psi \otimes \cchi$ 
is non-principal and hence $L(s, \psi \otimes \cchi)$ 
has an analytic continuation to the entire complex plane.  
It satisfies the functional equation
$$
\Lambda(s, \psi \otimes \cchi)
=
\epsilon(d)\Lambda(1-s, \bar {\psi} \otimes \cchi),
$$
where
\begin{equation}\label{fun}
\Lambda(s, \psi \otimes \cchi)
~=~
\pi^{-s/2}(8qd)^{s/2}\Gamma(s/2)L(s, \psi \otimes \cchi),
\phantom{m}
\epsilon(d)= \psi(8d)\Big(\frac{8d}{q}\Big)
\frac{\tau(\psi)}{\sqrt{q}}
\end{equation}
and $\tau(\psi)= \sum_{a=1}^q \psi(a)e(a/q)$ is the Gauss sum associated to
$\psi$ (see \cite{HD} and also eq 1.4 and page 13 of \cite{KST}).
For an integer $d$, let us set
\begin{equation}\label{approxA}
A(d, \psi) = \sum_{n=1}^{\infty}\Big(\frac{8d}{n}\Big)
\frac{d_\psi(n)}{\sqrt{n}}
V\Big(\frac{\pi n}{8dq}\Big),
\end{equation}
where 
$d_\psi(n)= \displaystyle\sum_{a, b \ge 1 \atop{ab=n}}\psi(a)\overline{\psi (b)}$
and for real numbers $c>1/2$ and $x>0$,
\begin{equation}\label{Vx}
V(x)
~=~ 
\frac{1}{2\pi i}\int_{(c)}
\Big(\frac{\Gamma(s/2+1/4)}{\Gamma(1/4)}\Big)^2x^{-s}\frac{ds}{s}
\phantom{m}\text{and}\phantom{m}
V(0) =1.
\end{equation}
We note that (see Lemma 2.1 of \cite{KST}) $V(x)$ is a real-valued smooth function
on $[0,\infty)$ and for any $\epsilon>0$, it satisfies
\begin{equation}\label{vbound}
V(x)=1 + O(x^{1/2-\epsilon}) \qquad {\rm{and}} \qquad 
V(x)\ll e^{-x/2}.
\end{equation}
From Lemma 2.2 and eq 1.7 of \cite{KST}, for any square-free positive integer
$d$ co-prime to $2q$, we then have
\begin{equation}\label{approx}
|L(1/2, \poc)|^2= 2 A(d, \psi).
\end{equation}
Henceforth $\sum'$ denotes that the sum is over square-free integers.
Throughout the article, the constants depend on $q$ and are not necessarily 
explicitly mentioned.
We record the following lemmas which will be used in subsequent
sections. 
\begin{lem}\label{charsum}
Let $u \leq X$ be an odd natural number. If $u$ is a square, then
\begin{equation*}
\DD \chi_{8d}(u)
~=~
\frac{X}{\zeta(2)}\prod_{p|2uq}\Big(\frac{p}{p+1}\Big)
~+~
O(u^{\frac{1}{4}}X^{\frac12}),
\end{equation*}
while if $u$ is not a square, then 
\begin{equation*}
\DD \chi_{8d}(u) 
= O(u^{\frac14}X^{\frac12}(\log X)^{\frac32}).
\end{equation*}
\end{lem}

\begin{proof} 
Since $\displaystyle\mu^2(d)=\sum_{a^2|d}\mu(a)$,
we get
\begin{equation}\label{chi}
\DD \chi_{8d}(u)
~=~ 
\displaystyle\sum_{X < d \le  2X \atop (d, 2q)=1} \Big(\frac{8d}{u}\Big)\sum_{a^2|d}\mu(a) 
~=~ \sum_{a \leq \sqrt{2X} \atop{(a,2q)=1}}
\mu(a)\Big(\frac{8a^2}{u}\Big) \sum_{\bb \atop{(b,2q)=1}}\Big(\frac{b}{u}\Big).
\end{equation}
When $u$ is a square, we get
\begin{eqnarray*}
\sum_{a \leq \sqrt{2X} \atop{(a, 2u q)=1}}
\mu(a)\sum_{\bb \atop{(b, 2 u q)=1}}1
&=&
\sum_{a \leq \sqrt{2X} \atop{(a, 2uq)=1}}\mu(a)
\Big[ \frac{X}{a^2}\prod_{p | 2 u q} (1- \frac{1}{p})
~+~ O(d(u))\Big]\\
&=&
\frac{X}{\zeta(2)}\prod_{p| 2u q}\Big(\frac{p}{p+1}\Big)
~+~O(d(u)X^{\frac12}),
\end{eqnarray*}
where $d(n)$ denotes the number of divisors of $n$.
Let $\chi_0$ be the principal character modulo $2q$. 
Then from \eqref{chi}, we have
\begin{equation}\label{chi1}
\DD \cchi(u)
= 
\Big(\frac{8}{u}\Big)\sum_{a \leq \sqrt{2X} \atop{(a, 2uq)=1}}
\mu(a)\sum_{\bb}\chi_0(b)\Big(\frac{b}{u}\Big).
\end{equation}
If $u$ is not a square, then $\chi_0(\cdot)\Big(\frac{\cdot}{u}\Big)$ is a non-principal 
character and using P\'olya-Vinogradov inequality, we get 
$$
\sum_{\bb}\chi_0(b)\Big(\frac{b}{u}\Big)
~\ll~
\min \Big( \sqrt{u}\log u,~ \frac{X}{a^2}\Big)
~\ll~
\Big( \sqrt{u}(\log u)\frac{X}{a^2}\Big)^{\frac12}.
$$
Therefore
\begin{equation*}
\DD \chi_{8d}(u) 
\ll 
u^{\frac14}(\log u)^{\frac12}X^{\frac12}
\sum_{a \leq \sqrt{2X} \atop{(a, 2uq)=1}}\frac{1}{a}\\
~\ll~
u^{\frac14}(\log u)^{\frac12}X^{\frac12}\log X
~\ll~ 
u^{\frac14}X^{\frac12}(\log X)^{\frac32}.
\end{equation*}
This completes the proof of \lemref{charsum}.
\end{proof}

We now introduce a few more notations which will be used throughout the article.
Let $\theta$ be a character modulo $m$, then for any integer $k$, define
$$
\tau_{k}(\theta) = \sum_{c \!\!\!\!\! \pmod{m}} \theta(c) e\Big(\frac{kc}{m} \Big).
$$
When $k=1$, we denote $\tau_{k}(\theta)$ by $\tau(\theta)$ and when
$\theta$ is the quadratic character $\Big(\frac{ \cdot }{m}\Big)$,
we denote $\tau_{k}(\theta)$ by $\tau_k(m)$. In this set-up,
we have the following lemma.
\begin{lem}[{\rm Theorem 9.6, page 287 of \cite{MV}}]\label{GaussS}
Let $k$ be an integer, $q_1, q_2$ be co-prime positive integers and 
$\theta_1, \theta_2$ be Dirichlet characters modulo $q_1$ and $q_2$ respectively. 
Then we have
$$
\tau_k(\theta_1\theta_2) = \theta_1(q_2)\theta_2(q_1)\tau_k(\theta_1)\tau_k(\theta_2).
$$
\end{lem}
Let $m$ be an odd integer. Following \cite{KST}, for any integer $k$, define 
$$
G_k(m) = \Big(\frac{1- i}{2} + \Big(\frac{-1}{m}\Big)\frac{1 + i }{2}\Big)\tau_k(m).
$$

\begin{lem}[{\rm Lemma 2.3 of \cite{KST}}]\label{Sound}
We have
\begin{enumerate}
\item
If $m, n$ are co-prime odd integers, then $G_k(mn) = G_k(m) G_k(n)$. 

\item
Let $\alpha$ be the largest positive integer such that $p^{\alpha} \mid k$. By convention, 
$\alpha = \infty$ when $k =0$. Then for $\beta \ge 1$, 
$$
G_k(p^{\beta}) = 
\begin{cases}
0 & \beta \le \alpha \textrm{~ is ~ odd}\\
\varphi(p^{\beta}) & \beta \le \alpha \textrm{~ is ~ even}\\
-p^{\alpha} & \beta = \alpha + 1 \textrm{~ is ~ even}\\
\Big(\frac{kp^{-\alpha}}{p}\Big)p^{\alpha}\sqrt{p} & \beta = \alpha + 1 \textrm{~ is ~ odd}\\
0 & \beta \ge \alpha + 2.
\end{cases}
$$
\end{enumerate}
\end{lem}
Let $q>1$ be an odd integer.
From now on, fix an even, non-quadratic Dirichlet character $\psi_0$ 
modulo $q$ such that $\psi_0^2$ is primitive of conductor $q$.
We define the resonator function as
$$
R(d)= \sum_{n \leq N}r(n)\cchi(n)\psi_0(n),
$$
where $N=X^{\frac{1}{24}}$
and $r(n)$ is a multiplicative function defined as follows.
For a prime $p$, set
$$
r(p)
=
\begin{cases}
\frac{L}{\sqrt{p}\log p} & {\rm if~} L^2 \leq p \leq L^4\\
0 & {\rm otherwise},
\end{cases}
$$
where $ L= \frac18\sqrt{\log N \log\log N}$
and set $r(p^n)=0$ for $n>1$.
We note that $|r(n)| \leq 1$ for all $n$.
In \S3, we will prove the following proposition.
\begin{prop}\label{brd}
Let the notations be as above. We have
\begin{equation*}
\DD|R(d)|^2 
~\leq~ \frac{X}{\zeta(2)} \prod_p (1+ r(p)^2)
~+~ 
 O(X^{\frac58}(\log X)^{\frac32}).
\end{equation*} 
Further 
\begin{equation*}
\DD |R(d)|^6 
~\ll~ 
X \exp{\Big(O\Big(\frac{\log X}{\log\log X}\Big)\Big)}.
\end{equation*}
\end{prop}
Using random matrix theory, Keating and Snaith \cite{KeSn} conjectured that
$$
\sideset{}{^{\flat}}\sum_{|d| \leq X}L(1/2, \chi_d)^k 
~\sim~ 
C_k X(\log X)^{\frac{k(k+1)}{2}},
$$
where $C_k$ is a positive constant depending on $k$.
Throughout $\sideset{}{^{\flat}}\sum$ implies that the sum is over 
fundamental discriminants $d$. In \cite{HB},
Heath-Brown showed that for any $\epsilon >0$
and $\sigma \geq  1/2$,
$$
\sideset{}{^{\flat}}\sum_{|d| \leq X}|L(\sigma+ it , ~\chi_d)|^4 
~\ll~ 
X^{1+ \epsilon}(1 + |t|)^{1 + \epsilon}.
$$
In \cite{RS}, Rudnick and Soundararajan
showed that the predicted lower bounds of Keating and Snaith
are true for the $k$-th moments of $|L(1/2, \chi_d)|$. From the work
of Harper \cite{AH} (see also Soundararajan \cite{KS09}),
conditionally on the Generalised Riemann Hypothesis, 
upper bounds of correct order are also known. 
A recent result of Chen \cite{QS} proves an asymptotic formula for the 
fourth moment of $L(1/2, \chi_d)$, conditionally on the Generalised 
Riemann hypothesis. To the best of our knowledge $k$-th moments
of $|L(1/2, \psi \otimes \cchi )|$ for $k \ge 3$
have not been studied in the literature.
However, for our purposes, we need a much weaker upper bound on the fourth moment 
of $|L(1/2, \psi \otimes \cchi )|$ which is analogous to the bound of Heath-Brown.
More precisely, in \S4, we prove the following proposition.
\begin{prop}\label{moment}
For any $\epsilon>0$, we have
$$
\DD |L(1/2, \psi \otimes \cchi )|^4
~\ll~
X^{1+ \epsilon}.
$$
\end{prop}
Finally, in \S5, we prove the following proposition which plays a
crucial role in proving \thmref{mainthm}.
\begin{prop}\label{M2}
Let $q>1$ be an odd integer and
$\psi$ be a non-quadratic, even Dirichlet character modulo $q$ such that
$\psi^2$ is a primitive character of conductor $q$.  
Also let $\mathcal{A}$ be a subset of a set of co-prime residue classes modulo $2q$ 
with cardinality at most $\varphi(q)/3$.
When $\psi = \psi_0$, we have
\begin{eqnarray*}
&&
\sideset{}{'}\sum_{\substack{X < d \leq 2X \atop { (d, 2q) =1  
\atop  {d \not\equiv a \!\!\!\! \!\pmod{2q} \text{ for all } 
a \in \mathcal{A}}}}}
|L(1/2,\psi_0 \otimes \cchi)|^2|R(d)|^2 \\
&\gg& 
X\log X \prod_{p} (1+ r(p)^2)\exp\Bigg(\Big(\frac{1}{8\sqrt{24}} + o(1)\Big)
\sqrt{\frac{\log X}{\log\log X}}\Bigg),
\end{eqnarray*}
while for $\psi \neq \psi_0, \bar{\psi}_0$, we have 
$$
\DD |L(1/2,\poc)|^2|R(d)|^2 
~\ll~
X\log X \prod_{p} (1+r(p)^2)\exp\Bigg(o\Big(\sqrt{\frac{\log X}{\log\log X}}\Big)\Bigg).
$$
\end{prop}

\begin{rmk}\label{mainrmk2}
Results analogous to \propref{brd} and \propref{M2} for modular $L$-functions
are obtained in \cite{GKS}. 
\end{rmk}

\section{Proof of \propref{brd}} 
By the definition of resonator function, we have
\begin{equation}\label{p3}
\DD |R(d)|^2 
=
\sum_{m,n \leq N}
r(m)r(n)\psi_0(m){\overline{\psi_0(n)}}\DD \cchi(mn).
\end{equation}
Since $r(n)$ is non-zero only when $n$ is square-free, we note that
$m, n$ are square-free in the above sum.
We now apply \lemref{charsum} to the inner sum of the right hand side
of \eqref{p3}. Contribution of the main term in $\DD\cchi(mn)$
occurs when $mn$ is a square. Since $m, n$ are square-free, $mn$
is a square only when $m=n$. Therefore, the main term of \eqref{p3} 
is
\begin{eqnarray*}
~ \leq ~
\frac{X}{\zeta(2)}\sum_{m \leq N}r(m)^2
\prod_{p|2qm}\Big(\frac{p}{p+1}\Big) ~+~ O(X^{\frac{5}{8}})
&\leq & 
\frac{X}{\zeta(2)}\sum_{m \leq N}r(m)^2
+ O(X^{\frac{5}{8}}) \\
&\leq &
\frac{X}{\zeta(2)} \pp \Big(1+r(p)^2\Big)
+ O(X^{\frac{5}{8}}).
\end{eqnarray*}
We bound the remaining terms of \eqref{p3} by
$$
~\ll~
X^{\frac12}(\log X)^{\frac32}
\sum_{m,n \leq N} (mn)^{\frac14} 
~\ll~
X^{\frac12}(\log X)^{\frac32}N^{\frac52}
~\ll~
X^{\frac58}(\log X)^{\frac32}.
$$
This completes the proof of the first part of \propref{brd}. 
Again we apply \lemref{charsum} to calculate 
$$
\DD |R(d)|^6 
~=~ 
\sum_{n_1, \ldots, n_6 \leq N} 
r(n_1)\ldots r(n_6)\psi_0(\prod_{i=1}^3 n_i)
~{\bar {\psi}_0(\prod_{i=4}^6 n_i)}
\DD \cchi(\prod_{i=1}^{6}n_i).
$$ 
When $\prod_{i=1}^6 n_i$ is a square, each prime dividing the product
divides an even number of $n_i$'s as $n_i$'s are square-free. 
Therefore, we get
\begin{eqnarray*}
\DD |R(d)|^6 
&\ll &
X
\displaystyle\sum_{n_1, \ldots, n_6 \leq N \atop{n_1\ldots n_6=\square}}r(n_1)\ldots r(n_6)  \\
& \ll &
X \pp \Big(1+ {6 \choose 2}r(p)^2 + {6 \choose 4}r(p)^4
+ {6 \choose 6}r(p)^6\Big)\\
& \ll & 
X \exp\Big( 45 \sum_{L^2 \leq p \leq L^4} 
\frac{L^2}{p(\log p)^2} \Big) 
~\ll~
X \exp\Big(O\Big(\frac{\log X}{\log\log X}\Big) \Big).
\end{eqnarray*}
The contribution of terms when $\prod_{i=1}^6 n_i$ is not a square 
is 
\begin{equation*}
\ll
X^{\frac12}(\log X)^{\frac32}
\sum_{n_1,\ldots,n_6 \leq N}
(n_1n_2n_3n_4n_5n_6)^{\frac14}
~\ll~
X^{\frac12}(\log X)^{\frac32}N^{\frac{15}{2}}
~\ll~
X^{\frac{13}{16}}(\log X)^{\frac32}.
\end{equation*}
This completes the proof of \propref{brd}.
\qed

\smallskip

\section{Proof of \propref{moment}}

We derive \propref{moment} from the following general 
proposition which may be of independent interest. 
\begin{prop}\label{inter}
Let $\psi, \cchi$ be as before and $s_0=\sigma_0+ it_0 \in \C$ with $1/2 \le \sigma_0 \le 1$,  
$T_0 = 1 + |t_0|$.  Then for any sufficiently large real number $X$
and any $\epsilon  >0$, we have
$$
\DD |L(s_0, \psi \otimes \cchi )|^2|L(s_0, \bar{\psi} \otimes \cchi)|^2
~\ll~
(X+(XT_0)^{2-2\sigma_0})(XT_0)^{\epsilon}.
$$
\end{prop}
We need the following lemma to prove \propref{inter}.

\begin{lem}{\rm (Heath-Brown, Corollary 3 of  \cite{HB})}\label{HB}~
Let $M, Q$ be positive integers and let $a_1, \ldots, a_M$ be
arbitrary complex numbers. Let $S(Q)$ denote the set of 
all real primitive characters of conductor at most $Q$.
Then for any $\epsilon >0$, we have
$$
\sum_{\chi \in S(Q)}\Big|\sum_{n \leq M}a_n \chi(n)\Big|^2
~\ll~
Q^{\epsilon}M^{1+\epsilon}(Q + M)\max_{n \le M}|a_n|^2.
$$ 
\end{lem}

\begin{proof}[Proof of \propref{inter}]
When $\sigma_0 =1$, the proposition follows from the well known inequality 
(see page 350, Lemma 10.15 of  \cite{MV})
$$
|L(1+it, \psi \otimes \cchi)| \ll \log(8dq(|t|+4)).
$$
Thus we can assume that $1/2 \le \sigma_0 <1$.
We follow the proof of Theorem 2 of \cite{HB}. Consider the integral
\begin{equation*}
I
~=~
\frac{1}{2\pi i}\int_{(2)}
L(s, \poc)L(s, \bar{\psi} \otimes \cchi)~\Gamma(s - s_0)~U^{s - s_0} ~ds
\end{equation*}
for a real parameter $U \ge 1$.
By Mellin inversion (see Page 99, Exercise 6.6.3 of \cite{MM}), we have
\begin{equation}\label{one}
I 
~=~
\sum_{n=1}^{\infty} \frac{\dsi(n)\cchi(n)}{n^{s_0}}e^{-\frac{n}{U}}.
\end{equation}
Moving the line of integration to $(\alpha)$ with $0 \leq \alpha < \sigma_0 < 1$, we get
\begin{eqnarray}\label{two}
I
~=~
L(s_0, \poc)L(s_0, \bar{\psi} \otimes \cchi)
+
\frac{1}{2\pi i}\int_{(\alpha)}L(s, \poc)L(s, \bar{\psi} \otimes \cchi)
\Gamma(s - s_0)U^{s - s_0}ds.
\end{eqnarray} 
Let $T = 1+ |t| $ and let $\nu(\beta)$ for $0 \le \beta \le 1$ be the infimum of all $\nu$ such that
\begin{equation}\label{nu}
\DD |L(\beta +it,\poc)|^2|L(\beta +it, \bar{\psi} \otimes \cchi)|^2 
~ \ll ~ 
(X+ (XT)^{2-2\beta})(XT)^{\nu}
\end{equation} 
uniformly in $X$ and $T$.
Applying Cauchy-Schwarz inequality, \eqref{one} and \eqref{two}, we get
\begin{eqnarray*}
|L(s_0, \poc)|^2|L(s_0, \bar{\psi} \otimes \cchi)|^2
&\ll&
\Big\vert\sum_{n=1}^{\infty} \frac{\dsi(n)\cchi(n)}{n^{s_0}}e^{-\frac{n}{U}}\Big\vert^2 \\
&+&
\Big\vert\int_{(\alpha)}L(s, \poc)L(s,  \bar{\psi} \otimes \cchi)
\Gamma(s - s_0)U^{s - s_0}ds\Big\vert^2.
\end{eqnarray*}
Again using Cauchy-Schwarz inequality and Stirling's formula, we get
\begin{eqnarray*}
&&
\Big\vert\int_{(\alpha)}L(s, \poc)L(s,  \bar{\psi} \otimes \cchi)
\Gamma(s - s_0)U^{s - s_0}ds\Big\vert^2 \\
&\ll &
U^{2(\alpha-\sigma_0)}\int_{-\infty}^{\infty}
|L(\alpha +it, \poc)|^2|L(\alpha +it,  \bar{\psi} \otimes \cchi)|^2e^{-|t-t_0|}dt.
\end{eqnarray*}
Therefore for $0 \leq \alpha < \sigma_0 < 1$, we have
\begin{eqnarray}\label{al1}
&&
\DD |L(s_0, \poc)|^2|L(s_0, \bar{\psi} \otimes \cchi)|^2 \nonumber \\
&\ll &
\DD \Big\vert\sum_{n=1}^{\infty} 
\frac{\dsi(n)\cchi(n)}{n^{s_0}}
e^{-\frac{n}{U}}\Big\vert^2 \nonumber \\
&+&
U^{2(\alpha-\sigma_0)}\int_{-\infty}^{\infty}
\DD |L(\alpha +it, \poc)|^2|L(\alpha +it,  \bar{\psi} \otimes \cchi)|^2
 e^{-|t-t_0|}dt. 
\end{eqnarray}
Applying Stirling's formula and functional equation \eqref{fun}, we get 
\begin{eqnarray}\label{al2}
|L(\alpha+it, \poc)| 
&\ll&
 d^{1/2-\alpha} T^{1/2-\alpha}|L(1-\alpha + it, \poc)| \\
 \phantom{m}\text{and}\phantom{m}
 |L(\alpha+ it, \bar{\psi} \otimes \cchi)| 
&\ll&
 d^{1/2-\alpha} T^{1/2-\alpha}|L(1-\alpha + it, \bar{\psi} \otimes \cchi). \nonumber
\end{eqnarray}
Applying \eqref{nu}, \eqref{al1} and \eqref{al2}, for any $\epsilon>0$, we have
\begin{eqnarray}\label{L4bound}
&&
\DD |L(s_0, \poc)|^2|L(s_0, \bar{\psi} \otimes \cchi)|^2    \nonumber \\
&\ll &
\DD \Big\vert\sum_{n=1}^{\infty} \frac{\dsi(n)\cchi(n)}{n^{s_0}}e^{-\frac{n}{U}}\Big\vert^2 
+ 
U^{2(\alpha-\sigma_0)}(XT_0)^{2-4\alpha}
(X ~+~ (XT_0)^{2\alpha})(XT_0)^{\nu(1-\alpha) + \epsilon}. \phantom{mmm}
\end{eqnarray}
Choosing $N_0= U\log ^2(XT_0)$, we note that
$$
\DD \Big\vert\sum_{n>N_0} 
\frac{\dsi(n)\cchi(n)}{n^{s_0}} e^{-\frac{n}{U}} \Big\vert^2
~\ll~
\DD \Big( \int_{N_0}^{\infty}e^{-y/U} dy \Big)^2
~\ll~
XU^2e^{-2N_0/U}.
$$
We split the sum 
$$
\sum_{n\le N_0} \frac{\dsi(n)\cchi(n)}{n^{s_0}}e^{-\frac{n}{U}}
$$
in $O(\log N_0)$ intervals of the form $M/2 < n \le M$ with $M \ll N_0$.
Now, using \lemref{HB}, we see that 
$$
\DD  \Big\vert \sum_{M/2< n \le M} 
\frac{\dsi(n)\cchi(n)}{n^{s_0}}e^{-\frac{n}{U}} \Big\vert^2
~\ll~
X^{\epsilon}M^{1-2\sigma_0 +\epsilon}(X + M).
$$
Therefore, applying Cauchy-Schwarz inequality, for any $1/2 \leq \sigma_0 < 1$, we get 
\begin{eqnarray*}
\DD \Big\vert\sum_{n=1}^{\infty} 
\frac{\dsi(n)\cchi(n)}{n^{s_0}}e^{-\frac{n}{U}}\Big\vert^2
&\ll &
(XT_0U)^{\epsilon}(X+ U^{2-2\sigma_0})
~+~
XU^2e^{-2\log^2{(XT_0)}}.
\end{eqnarray*}
Putting this in \eqref{L4bound}, for $1/2 \le \sigma_0 < 1$ and $0 < \alpha < \sigma_0$, we have 
\begin{eqnarray}\label{est}
\DD |L(s_0, \poc)|^2|L(s_0, \bar{\psi} \otimes \cchi)|^2 
&\ll &
(XT_0U)^{\epsilon}(X+ U^{2-2\sigma_0})
~+~
XU^2e^{-2\log^2{XT_0}}  \nonumber \\ 
&+& 
U^{2(\alpha-\sigma_0)}
(X + (XT_0)^{2\alpha})(XT_0)^{2-4\alpha + \nu(1-\alpha) + \epsilon}.
\phantom{mmm}
\end{eqnarray}
For $\sigma_0 > 1/2$, choosing $\alpha = 1-\sigma_0$ and $U~=~(XT_0)^{1+\delta}$
for any $0< \delta <1$ in \eqref{est}, we have 
\begin{eqnarray*}
&&
\DD |L(s_0, \poc)|^2 |L(s_0, \bar{\psi} \otimes \cchi)|^2\nonumber \\
&\ll &
(XT_0)^{\epsilon}(X+ (XT_0)^{(2-2\sigma_0)})((XT_0)^{\delta}
~+~
(XT_0)^{(2-4\sigma_0)\delta+\nu(\sigma_0)})
\end{eqnarray*}
uniformly in $t_0$.
Since $\epsilon >0$ is arbitrary, by \eqref{nu}, we have
$$
\nu(\sigma_0) \leq \max\{ \delta,~ \nu(\sigma_0)-(4\sigma_0-2)\delta \}.
$$
Since $1/2 < \sigma_0 < 1$,  we get $\nu(\sigma_0) \le \delta$. Since $\delta$ 
can be chosen arbitrarily small, we get $\nu(\sigma_0) \le 0$.
Now let $\sigma_0=1/2$. Choosing $\alpha = 1/2 - \epsilon$ in \eqref{est} 
and using  the fact that $\nu(1/2+\epsilon) \leq 0$ from above, we get 
\begin{eqnarray}
&&
\DD |L(s_0, \poc)|^2|L(s_0, \bar{\psi} \otimes \cchi)|^2 \nonumber \\
&\ll &
(XT_0U)^{\epsilon}(X+ U)
~+~
XU^2e^{-2\log^2{XT_0}}
~+~~
U^{-2\epsilon}
(X ~+~ (XT_0)^{1-2\epsilon})(XT_0)^{5\epsilon}.
\end{eqnarray}
Finally, choosing $U=XT_0$ and noting that $\epsilon >0 $ can be chosen arbitrarily 
small, we get $\nu(1/2)\le 0$. This completes the proof of \propref{inter}.
\end{proof}
As an immediate consequence of \propref{inter}, we get
$$
\DD |L(1/2, \psi \otimes \cchi )|^4
~\ll~
X^{1+ \epsilon}.
$$
This completes the proof of \propref{moment}.
\qed

\section{Proof of \propref{M2}}
Let $\psi, \psi'$ be arbitrary primitive, non-quadratic Dirichlet characters  
of odd conductor $q >1$ and
$\displaystyle d_{\psi}(n)=\sum_{n=ab}\psi(a)\overline{\psi (b)}$.
Applying prime number theorem and eq 5.79 on page 124 of \cite{IK}, 
we have 
\begin{equation}\label{psisum1}
\sum_{p \leq x}d_{\psi}(p)^2 
~=~ 
(2 + o(1)) \frac{x}{\log x}
\end{equation}
and  for $\psi \neq \psi', \bar{\psi'}$, we have 
\begin{equation}\label{psisum2}
\sum_{p \leq x}d_{\psi}(p)d_{\psi'}(p)
~=~
o\Big( \frac{x}{\log x} \Big).
\end{equation}
Let $\Phi(t)$ be a smooth function supported on $[1/2, 5/2]$ and 
$\hat{\Phi}(0)=\int_0^{\infty}\Phi(t)dt$. In this set-up, we have the following 
Lemma.

\begin{lem}\label{Mainterm}
Let $q >1$ be an odd integer and $\psi$ be a non-quadratic, 
even Dirichlet character modulo $q$ such that $\psi^2$ is primitive modulo $q$.
Also let $u$ be an integer with $(u, 2q)=1$ and
 $h, h_{\psi}$ be multiplicative functions defined by
$$
h(p^k)=1+\frac{1}{p}+\frac{1}{p^2}-\frac{d_{\psi}(p)^2}{p(p+1)}
$$
and
$$
h_{\psi}(p^k) = 1 + \frac{1}{p} + \frac{\psi(p)^2}{p^2} - \frac{ 4 \psi(p)^2}{p(p + \psi(p)^2 )}
$$
respectively for primes $p$ and integers $k \geq 1$.
Write $u=u_1u_2^2$, where $u_1$ is square-free.
Then for any co-prime residue class $a\!\!\pmod {2q}$ and
for any sufficiently small $\epsilon>0$, we have 
\begin{eqnarray}\label{nt}
&&
\sideset{}{'}\sum_{d \atop{d \equiv a\!\! \!\!\!\pmod {2q}}}
\cchi(u)|L(1/2,\poc)|^2\Phi\Big(\frac{d}{X}\Big)  \\
&=&
~\frac{A_{\psi}\hat{\Phi}(0)}{\varphi(q)}\frac{\dsi(u_1)}{h(u)}
\frac{\sqrt{u_1}}{\sigma(u_1)}X   \Big(\log \frac{D_1X}{u_1} 
~+~
\sum_{p|u_1}\frac{D_2(p)}{p}\log p
~+~
 \sum_{p|u \atop{p\nmid u_1}}\frac{D_3(p)}{p}\log p \Big) \nonumber \\
&+&
\frac{2\hat{\Phi}(0)X}{\varphi(q)}\Bigg(C_{\psi}\overline{\psi(a)}\Big(\frac{\pm q}{a}\Big)
\frac{d(u_1)}{h_{\psi}(u)}
\frac{ \psi(u_1) \sqrt{u_1}}{ \prod_{p| u_1}( p + \psi(p)^2)} 
~+~
\overline{C_{\psi}}\psi(a)\Big(\frac{\pm q}{a}\Big) \frac{d(u_1)}{\overline{h_{\psi}(u)}}
\frac{\overline{\psi(u_1)} \sqrt{u_1}}{\prod_{p| u_1}(p + \overline{\psi(p)}^2)} \Bigg) \nonumber \\
&+& 
O\Big(u^{1/2+\epsilon}X^{3/4+\epsilon} + X^{7/8 + \epsilon} \Big), \nonumber
\end{eqnarray}
where $d(n)$ denotes the number of positive divisors of $n$, $\sigma(n)$ denotes the 
sum of positive divisors of $n$, 
$$
A_{\psi}
~=~
\frac{\varphi(q)}{2q\zeta(2)}|L(1,\psi^2)|^2
\Big|1-\frac{\psi(2)^2}{2}\Big|^2
\prod_{p\nmid 2q}\Big(1-\frac{1}{p}\Big)h(p)
\prod_{p|2q}\Big(\frac{p}{p+1}\Big)
$$ 
and 
\begin{eqnarray*}
C_{\psi} 
& = & 
(-1)^{\frac{q+1}{2}}\eta_{q} \frac{\varphi(q)}{2q^2}
\chi_{\pm q}(2)
\psi(2)L'(0, \overline{\psi}^2) L(1, \psi^2)^2
 \tau(\psi \otimes \chi_{\pm q}) 
 \Big(1-\frac{\psi(2)^2}{2}\Big)^2 \Big(1 - 2\overline{\psi(2)}^2 \Big) \\
&&
\prod_{p \nmid 2q} 
\Big(1 - \frac{1}{p}\Big)\Big(1 - \frac{\psi(p)^4}{p^2}\Big) h_{\psi}(p) . 
\end{eqnarray*}
Here $\chi_{\pm q}=\Big(\frac{(-1)^{\frac{q-1}{2}} q}{\cdot}\Big)$ is a 
quadratic character modulo $q$,
$$
\eta_q 
~=~
\begin{cases}
1&  \text{ if  }~~q \equiv 1 \pmod{4}\\
i &  \text{ if  }~~q \equiv 3 \pmod{4},
\end{cases}
$$
$D_1$ is a positive constant which depends on $\Phi, \psi$ and 
$D_2(p), D_3(p)$ are $\ll 1$. 
We further note that taking sum over all co-prime residue class $a\!\!\pmod {2q}$
in \eqref{nt}, the second principal term vanishes and we get
\begin{eqnarray}\label{nt2}
&&
 \sideset{}{'}\sum_{d \atop (d,2q)=1} 
\cchi(u)|L(1/2,\poc)|^2\Phi\Big(\frac{d}{X}\Big) \\
& = &
A_{\psi}\hat{\Phi}(0)\frac{\dsi(u_1)}{h(u)}
\frac{\sqrt{u_1}}{\sigma(u_1)}X   \Big(\log \frac{D_1X}{u_1} 
~+~
\sum_{p|u_1}\frac{D_2(p)}{p}\log p
~+~
 \sum_{p|u \atop{p\nmid u_1}}\frac{D_3(p)}{p}\log p \Big) \nonumber \\
 &+&
O\Big(u^{1/2+\epsilon}X^{3/4+\epsilon} +X^{7/8 + \epsilon} \Big) \nonumber. 
\end{eqnarray}
\end{lem}

\begin{proof}
Note that 
$$
\mu^2(d)
~=~
 \sum_{ v^2 |d}\mu(v)
~=~ 
\sum_{v^2|d \atop{v \le X^{1/8}}}\mu(v)
~+~ 
\sum_{v^2|d \atop{v > X^{1/8}}}\mu(v).
$$
Applying \eqref{approx} and the above observation, 
we can split the sum in \eqref{nt} as
\begin{align}\label{split}
\sideset{}{'}\sum_{d \atop{d \equiv a\!\! \!\!\!\pmod {2q}}} 
\cchi(u)|L(1/2,\poc)|^2\Phi\Big(\frac{d}{X}\Big)
&=
2\sideset{}{'}\sum_{d \atop{d \equiv a\!\! \!\!\!\pmod {2q}}} 
\cchi(u) A(d, \psi) \Phi\Big(\frac{d}{X}\Big) \nonumber \\
&=~
2 \sum_{d \atop{d \equiv a\!\! \!\!\!\pmod {2q}}} 
\sum_{v^2|d \atop{v \le X^{1/8}}}\mu( v)
\cchi(u) A(d, \psi)\Phi\Big(\frac{d}{X}\Big)\nonumber \\
&+~
2 \sum_{d \atop{d \equiv a\!\! \!\!\!\pmod {2q}}}
\sum_{v^2|d \atop{v > X^{1/8}}}\mu(v)
\cchi(u)A(d, \psi)\Phi\Big(\frac{d}{X}\Big). 
\end{align}
Since
$\displaystyle \sum_{v^2|d \atop{v > X^{1/8}}}\mu(v)  \ll d^{\epsilon}$,  
we see that $\displaystyle\sum_{d \atop{d \equiv a \!\!\pmod {2q}}}
\sum_{v^2|d \atop{v > X^{1/8}}}\mu(v)
\cchi(u)A(d, \psi)\Phi\Big(\frac{d}{X}\Big)$ is
\begin{equation}\label{last}
\ll \!
\sum_{d \atop {(d, 2q)=1 \atop{\exists  v >  X^{1/8} \text{such that } v^2|d }}}
d^{\epsilon} |A(d, \psi)|\Phi\Big(\frac{d}{X}\Big) 
~\ll~
X^{\epsilon} 
\sum_{ X^{1/8}< \tilde{v} \leq \sqrt{5X/2} \atop{(\tilde{v}, 2q)=1}}
\sideset{}{'}\sum_{\frac{X}{2\tilde{v}^2}\leq m \leq \frac{5X}{2\tilde{v}^2} \atop (m, 2q)=1}
|A(\tilde{v}^2m, \psi)|.
\end{equation}
We bound the inner most sum as follows
\begin{equation}\label{eqN}
\sideset{}{'}\sum_{\frac{X}{2\tilde{v}^2}\leq m \leq \frac{5X}{2\tilde{v}^2} \atop (m, 2q)=1}
|A(\tilde{v}^2m, \psi)|
~\ll~
\frac{X^{1/2}}{\tilde{v}}\Bigg(\sideset{}{'}\sum_{\frac{X}{2\tilde{v}^2}\leq m 
\leq \frac{5X}{2\tilde{v}^2} \atop (m, 2q)=1}
|A(\tilde{v}^2m, \psi)|^2 \Bigg)^{1/2}.
\end{equation}
Dividing the sum over $m$ in three parts and then applying the bound of 
equation (3.5) on page~20 and Lemma 2.5 on page 15 of \cite{KST}, we get
\begin{equation}\label{eqNN}
\sideset{}{'}\sum_{\frac{X}{2\tilde{v}^2}\leq m 
\leq \frac{5X}{2\tilde{v}^2} \atop (m, 2q)=1}
|A(\tilde{v}^2m, \psi)|^2 
~\ll~ 
\frac{X^{1 + \epsilon}}{\tilde{v}^2}.
\end{equation}
Combining \eqref{eqN}, \eqref{eqNN} and putting it
in \eqref{last}, we get
\begin{eqnarray}\label{last1}
\sum_{d \atop{d \equiv a\!\! \!\!\!\pmod {2q}}}\sum_{v^2|d \atop{v >  X^{1/8}}}\mu(v)
\cchi(u) A(d, \psi) \Phi\Big(\frac{d}{X}\Big) 
&\ll&
X^{1+\epsilon}
\sum_{ X^{1/8}< \tilde{v} \leq \sqrt{5X/2}}
\frac{1}{\tilde{v}^2}
~\ll~ 
X^{7/8+\epsilon}.
\end{eqnarray}
We now estimate the sum $\displaystyle \sum_{d \atop{d \equiv a \!\!\pmod {2q}}} 
\sum_{v^2|d \atop{v \le X^{1/8}}}\mu( v)
\cchi(u) A(d, \psi)\Phi\Big(\frac{d}{X}\Big)$. This is 
analogous to Proposition 1.4 of \cite{KST} and Proposition 1.3 of
\cite{KS00}. So we present a sketch of the proof here. Using orthogonality of characters,
we can write 
\begin{eqnarray}\label{Cong}
&&
\sum_{d \atop{d \equiv a\!\! \!\!\!\pmod {2q}}} 
\sum_{v^2|d \atop{v \le X^{1/8}}}
\mu(v) 
\sum_{n=1}^{\infty}  \Big(\frac{8d}{un}\Big)
\frac{d_\psi(n)}{\sqrt{n}}
V\Big(\frac{\pi n}{8dq}\Big)
\Phi\Big(\frac{d}{X}\Big) \nonumber \\
&=&
\sum_{n=1}^{\infty} 
\frac{d_\psi(n)}{\sqrt{n}}\Big(\frac{8}{un}\Big)
\sum_{d \atop{d \equiv a\!\! \!\!\!\pmod {2q}}} 
\sum_{v^2|d \atop{v \le X^{1/8}}}
\mu(v)  \Big(\frac{d}{un}\Big)
V\Big(\frac{\pi n}{8dq}\Big)
\Phi\Big(\frac{d}{X}\Big) \nonumber \\
&=&
\frac{1}{\varphi(2q)}
\sum_{\theta \!\!\! \!\! \pmod{2q}} \overline{\theta (a)}
\sum_{n=1}^{\infty} 
\frac{d_\psi(n)}{\sqrt{n}} \Big(\frac{8}{un}\Big)
\sum_d  \sum_{v^2|d \atop{v \le X^{1/8}}}
\mu(v) \Big(\frac{d}{un}\Big)  \theta(d)
V\Big(\frac{\pi n}{8dq}\Big)
\Phi\Big(\frac{d}{X}\Big). \phantom{mm}
\end{eqnarray}
Note that the terms involving $n$ contribute only when $(n, 2q)=1$.
For a fixed positive integer $n$, define 
$$
\displaystyle F_n(t) = V\Big(\frac{\pi n}{8qXt}\Big) \Phi (t).
$$
Consider the sum
\begin{equation}\label{Pos1}
\sum_d  \sum_{v^2|d \atop{v \le X^{1/8}}}
\mu(v) \Big(\frac{d}{un}\Big)  \theta(d)
F_n\Big(\frac{d}{X}\Big)
~=~
\sum_{v \le X^{1/8} \atop{ (v, un ) =1}} \mu(v) \theta(v^2)
\sum_{d} \Big(\frac{d}{un}\Big)\theta(d) 
F_n\Big(\frac{v^2d}{X} \Big).
\end{equation}
Since $(un, 2q)=1$, $\theta(\cdot) \Big(\frac{\cdot}{un}\Big)$ is a character modulo $2q un$. 
Using Poisson summation formula and \lemref{GaussS}, we get 
\begin{eqnarray}\label{Pos2}
\sum_{d} \Big(\frac{d}{un}\Big)\theta(d) 
F_n\Big(\frac{v^2d}{X} \Big) 
&=&
\sum_{b \!\!\!\!\! \pmod{2qun}}
\Big(\frac{b}{un}\Big)\theta(b)
~~\sum_{m}
F_n\Big(\frac{v^2}{X}(b + 2mqun) \Big) \nonumber \\
&=&
\frac{X}{2q unv^2} \sum_{b \!\!\!\!\! \pmod{2qun}}
\Big(\frac{b}{un}\Big)\theta(b)
~~\sum_{k} 
\widehat{F}_n\Big(\frac{kX}{2qunv^2}\Big) e\Big(\frac{kb}{2qun} \Big)
\nonumber\\
&=&
\frac{X}{2q unv^2}
\sum_{k} 
\widehat{F}_n\Big(\frac{kX}{2qunv^2}\Big)
\sum_{b \!\!\!\!\! \pmod{2qun}}
\Big(\frac{b}{un}\Big)\theta(b) e\Big(\frac{kb}{2qun} \Big)
\nonumber\\
&=&
\frac{X}{2q unv^2}
\sum_{k} 
\widehat{F}_n\Big(\frac{kX}{2qunv^2}\Big)
\theta(un)\Big(\frac{2q}{un} \Big)
\tau_k(\theta) \tau_k(un).
\phantom{mm}
\end{eqnarray}
Applying \eqref{Pos2} in \eqref{Pos1} and then inserting it
in \eqref{Cong}, we get

\begin{eqnarray}\label{CongF}
&&
\sum_{d \atop{d \equiv a\!\! \!\!\!\pmod {2q}}} 
\sum_{v^2|d \atop{v \le X^{1/8}}}\mu( v)
\cchi(u) A(d, \psi)\Phi\Big(\frac{d}{X}\Big) 
\nonumber \\
&=&
\frac{X}{2qu\varphi(2q)}
\sum_{n=1}^{\infty} 
\frac{d_\psi(n)}{n^{3/2}} \Big(\frac{16q}{un}\Big)
\sum_{\theta \!\!\! \!\! \pmod{2q}} \overline{\theta (a)}
\theta(un)
\sum_{v \le X^{1/8} \atop{ (v, un ) =1}} 
\frac{\mu(v)\theta(v^2)}{v^2} \nonumber \\
&& 
\phantom{mmmmm}
\sum_{k} 
\widehat{F}_n\Big(\frac{kX}{2qunv^2}\Big) 
\tau_k(\theta)  \tau_k(un). 
\end{eqnarray}
We divide the right hand side of \eqref{CongF}
in two parts as follows
\begin{equation}\label{MandE}
\frac{X}{\varphi(2q)}
\sum_{n=1}^{\infty} 
\frac{d_\psi(n)}{\sqrt{n}} \Big(\mathcal{M}(u,n) + \mathcal{E}(u,n)\Big),
\end{equation}
where $\mathcal{M}(u,n)$ consists of terms when $k=0$ 
and the remaining terms are denoted by $\mathcal{E}(u,n)$.
We see that 
$$
\mathcal{M}(u,n)
~=~
\frac{1}{2qun}\Big(\frac{16q}{un}\Big)
\sum_{\theta \!\!\! \!\! \pmod{2q}} \overline{\theta(a)}\theta(un)
\sum_{v \le X^{1/8} \atop{ (v, un ) =1}} 
\frac{\mu(v)\theta(v^2)}{v^2} 
\widehat{F}_n(0) 
 \tau_0(un)
 \sum_{c \!\!\!\!\! \pmod{2q}}
\theta(c).
$$
Using orthogonality of characters, we see that the last sum in the above expression
is non-zero only when
$\theta$ is the principal character. Since $(un, 2q)=1$, we get
\begin{eqnarray*}
\mathcal{M}(u,n)
& = &
\frac{\varphi(2q)}{2qun}\Big(\frac{16q}{un}\Big)\widehat{F}_n(0) 
 \tau_0(un)
\sum_{v \le X^{1/8} \atop{ (v, 2qun ) =1}} 
\frac{\mu(v)}{v^2}\\
 & = &
\frac{\varphi(2q)}{2qun}\Big(\frac{16q}{un}\Big)\widehat{F}_n(0) \tau_0(un)
\Bigg( \frac{1}{\zeta(2)}\prod_{p|2qun} \Bigg(1 - \frac{1}{p^2}\Bigg)^{-1}
 ~+~O(X^{-1/8})\Bigg)\\
 & = &
\frac{\widehat{F}_n(0)}{\zeta(2)}\frac{\tau_0(un)}{\varphi(un)}
\Big(\frac{16q}{un}\Big) 
\prod_{p|2qun} \Bigg(\frac{p}{p+1}\Bigg)
 \Big(1~+~O(X^{-1/8})\Big).
\end{eqnarray*}
Note that $\tau_0(un) = 0$ unless $un$ is a perfect square, in which case, 
$\tau_0(un) = \varphi(un)$.
This implies that
$$
\sum_{n=1}^{\infty} 
\frac{d_\psi(n)}{\sqrt{n}}
{\mathcal M}(u,n)
= 
\frac{\Big(1~+~O(X^{-1/8})\Big)}{\zeta(2)}
\Bigg\{
\sum_{n=1 \atop{un \textrm{~is ~ a ~ square~}}}^{\infty} 
\frac{d_\psi(n)}{\sqrt{n}}
\widehat{F}_n(0) \Big(\frac{16q}{un}\Big) 
\prod_{p|2qun} \Bigg(\frac{p}{p+1}\Bigg)
\Bigg\}.
$$
This can be compared with equation (5.5) on page 27 of \cite{KST}.
It now follows from pages 27, 28, 29 of \cite{KST} that
\begin{eqnarray}\label{finmtest}
&&
\frac{X}{\varphi(2q)}
\sum_{n \ge 1}\frac{d_\psi(n)}{\sqrt{n}} 
{\mathcal M}(u,n)   \\
&=&
\frac{A_{\psi}\hat{\Phi}(0)}{2\varphi(q)}\frac{\dsi(u_1)}{h(u)}
\frac{\sqrt{u_1}}{\sigma(u_1)}X   \Big(\log \frac{D_1X}{u_1} 
~+~
\sum_{p|u_1}\frac{D_2(p)}{p}\log p
~+~
 \sum_{p|u \atop{p\nmid u_1}}\frac{D_3(p)}{p}\log p \Big) 
~+~
O\Big(X^{7/8 + \epsilon} \Big),\nonumber 
\end{eqnarray}
where $h, A_{\psi}, \hat{\Phi}$, $D_1$, $D_2(p)$ and $D_3(p)$ are 
as in the statement of \lemref{Mainterm}.
Now we estimate $\mathcal{E}(u,n)$ which is equal to
\begin{eqnarray}\label{newtau}
\frac{1}{2qun}\Big(\frac{16q}{un}\Big)
\!\!\!\sum_{\theta \!\!\! \!\! \pmod{2q}} \overline{\theta(a)}\theta(un)
\!\!\!\sum_{v \le X^{1/8} \atop{ (v, un ) =1}} 
\frac{\mu(v)\theta(v^2)}{v^2}\sum_{k \neq 0} 
\widehat{F}_n\Big(\frac{kX}{2qunv^2}\Big) 
 \tau_k(\theta) \tau_k(un).
\end{eqnarray}
From now onwards, for any character $\theta \!\!\pmod{2q}$, we denote its 
contribution to $\mathcal{E}(u,n)$ by $\mathcal{E}(u,n, \theta)$. We shall
consider two cases, one when $\psi \otimes \theta$ or $\overline\psi \otimes \theta$ is a quadratic character
and another when neither of them is a quadratic character.

\subsection*{Case I}
Suppose that neither $\psi \otimes \theta$ nor $\overline\psi \otimes \theta$ is 
a quadratic character.  
When $n > X^{1 + \epsilon}$, proceeding as in Lemma 5.2 of \cite{KST}, 
we get
\begin{eqnarray}
\mathcal{E}(u, n, \theta)
& \ll &
\frac{1}{un}
\sum_{v \le X^{1/8} \atop{ (v, un ) =1}} 
\frac{1}{v^2}\sum_{k \neq 0} 
\Big|\widehat{F}_n\Big(\frac{kX}{2qunv^2}\Big)\Big| 
 |\tau_k(un)|
~ \ll~
u^2X^{3/8}\exp\Big(-\frac{n}{40 qX}\Big).
\end{eqnarray}
Applying the bound $|d_{\psi}(n)| \ll n^{\epsilon}$, we see that
$$
\sum_{n=1}^{\infty} 
\frac{d_\psi(n)}{\sqrt{n}} \mathcal{E}(u, n, \theta)
~=~ 
\sum_{n \le X^{1 + \epsilon}} \frac{d_{\psi}(n)}{\sqrt{n}}  \mathcal{E}(u,n, \theta) 
~+~ 
O\Big(u^2 X^{-10}\Big).
$$
When $n \le  X^{1 + \epsilon}$, we can write $\mathcal{E}(u, n, \theta) 
= \mathcal{E}_1(u,n, \theta) + \mathcal{E}_2(u,n, \theta)$,
where $\mathcal{E}_1(u, n, \theta)$ consists of terms with
$0 < |k| \le nuX^{-3/4 + \epsilon}$  and 
$\mathcal{E}_2(u, n, \theta)$ consists of terms with
$|k| \ge nuX^{-3/4 + \epsilon}$. Proceeding as in
Lemma 5.2 of \cite{KST}, we observe that $\mathcal{E}_2(u,n, \theta)$ is bounded 
above by $u^2 X^{-9}$. Hence
\begin{eqnarray*}
\sum_{n=1}^{\infty} 
\frac{d_\psi(n)}{\sqrt{n}} \mathcal{E}(u,n, \theta)
&=& 
\frac{1}{2qu} \Big(\frac{16q}{u}\Big)
\overline{\theta(a)} \theta(u) 
\sum_{n \le X^{1 + \epsilon}} \frac{d_{\psi}(n)}{n^{3/2}}  \Big(\frac{16q}{n}\Big)
\theta(n)
\sum_{v \le X^{1/8} \atop{ (v, un ) =1}} 
\frac{\mu(v)\theta(v^2)}{v^2}\\
&&
\sum_{0< |k| \le nuX^{-3/4+ \epsilon}} 
\widehat{F}_n\Big(\frac{kX}{2qunv^2}\Big) 
\tau_k(\theta) \tau_k(un) ~+~ O\Big(u^2 X^{-8}\Big)\\
&=& 
\frac{1}{2qu} \Big(\frac{16q}{u}\Big)
\overline{\theta(a)} \theta(u) \sum_{v \le X^{1/8} \atop{ (v, u ) =1}} 
\frac{\mu(v)\theta(v^2)}{v^2}\sum_{0< |k| \le uX^{1/4+ \epsilon}}
\tau_k(\theta)\\
&&
\sum_{\frac{|k|}{u} X^{3/4- \epsilon} \le n \le X^{1 + \epsilon} \atop{(v,n) = 1}} 
\frac{d_{\psi}(n)}{n^{3/2}}  \Big(\frac{16q}{n}\Big)
\theta(n) 
\widehat{F}_n\Big(\frac{kX}{2qunv^2}\Big) 
 \tau_k(un)
+ O\Big(u^2 X^{-8}\Big)\\
& \ll & 
\frac{1}{u} 
\sum_{v \le X^{1/8} \atop{ (v, 2qu ) =1}} 
\frac{1}{v^2}
\sum_{0< |k| \le uX^{1/4+ \epsilon}}|E(k, v,\theta)|
+ O\Big(u^2 X^{-8}\Big),
\end{eqnarray*}
where 
$$
E(k, v,\theta)
~=~
\sum_{\frac{|k|}{u} X^{3/4- \epsilon} \le n \le X^{1 + \epsilon} \atop{(v,n) = 1}} 
\frac{d_{\psi}(n)}{n^{3/2}}  \Big(\frac{16q}{n}\Big)
\theta(n) 
\widehat{F}_n\Big(\frac{kX}{2qunv^2}\Big)
 \tau_k(un).
$$
Arguing as in pages $31$ to $35$ of \cite{KST}, we see that
\begin{equation}
\sum_{n=1}^{\infty} 
\frac{d_\psi(n)}{\sqrt{n}} \mathcal{E}(u,n, \theta)
~\ll~
\frac{u^{1/2 + \epsilon}}{X^{1/4 - \epsilon}}.
\end{equation}
Hence 
\begin{equation}\label{finerrorest}
\frac{X}{\varphi(q)} \sum_{\theta \atop{ \psi\otimes\theta, \bar{\psi}\otimes\theta \text { are not quadratic }}}
\sum_{n=1}^{\infty} 
\frac{d_\psi(n)}{\sqrt{n}} \mathcal{E}(u,n, \theta)
~\ll~
u^{1/2 + \epsilon} X^{3/4 + \epsilon}.
\end{equation}

\subsection*{Case II}
Suppose that at least one of $\psi \otimes \theta, ~\overline\psi \otimes \theta$ is a 
quadratic character.  Since $\psi^2$ has conductor $q$,
so does $\theta^2$ and hence $\theta$. Let $\theta^*$ be the primitive character modulo $q$
which induces $\theta$. Note that $\theta = \theta^* \otimes \chi_0$, where $\chi_0$ is
the principal character modulo $2q$.
Recall that (see the paragraph following \eqref{newtau})
\begin{eqnarray}\label{P30}
\mathcal{E}(u, n, \theta) = \frac{1}{2qun}\Big(\frac{16q}{un}\Big)
\overline{\theta(a)}\theta(un)
\!\!\!\sum_{v \le X^{1/8} \atop{ (v, un ) =1}} 
\frac{\mu(v)\theta(v^2)}{v^2}\sum_{k \neq 0} 
\widehat{F}_n\Big(\frac{kX}{2qunv^2}\Big) 
 \tau_k(un) \tau_{k}(\theta).
\end{eqnarray}
Since $\tau_{-k}(\theta) = \theta(-1)\tau_{k}(\theta)$
and $G_{-k}(n) = \Big(\frac{-1}{n}\Big)G_{k}(n)$, we have
\begin{eqnarray*}
&&
\sum_{k \neq 0} 
\widehat{F}_n\Big(\frac{kX}{2qunv^2}\Big) 
 \tau_k(un)
\tau_{k}(\theta)\\
& = &
\sum_{k \neq 0} 
\widehat{F}_n\Big(\frac{kX}{2qunv^2}\Big) 
 \Big(\frac{1+i}{2}\Big)G_k(un)\tau_{k}(\theta) 
+ 
 \sum_{k \neq 0} 
\widehat{F}_n\Big(\frac{kX}{2qunv^2}\Big) \Big(\frac{1-i}{2}\Big)G_{-k}(un)
\tau_{k}(\theta)\\
& = &
\sum_{k \neq 0} 
\widehat{F}_n\Big(\frac{kX}{2qunv^2}\Big) 
 \Big(\frac{1+i}{2}\Big)G_k(un)
 \tau_{k}(\theta) 
+ 
 \sum_{k \neq 0} 
\widehat{F}_n\Big(\frac{-kX}{2qunv^2}\Big) \Big(\frac{1-i}{2}\Big)G_{k}(un)
\tau_{-k}(\theta)\\
& = &
\sum_{k \neq 0} 
\Bigg( \Big(\frac{1+i}{2}\Big)\widehat{F}_n\Big(\frac{kX}{2qunv^2}\Big) 
 +  \theta(-1)\Big(\frac{1-i}{2}\Big)\widehat{F}_n\Big(\frac{-kX}{2qunv^2}\Big)\Bigg)
 G_k(un)
\tau_{k}(\theta)\\
 & = &
\sum_{k \neq 0} 
\widetilde{F}_{n, \theta}\Big(\frac{kX}{2qunv^2}\Big)G_k(un)
\tau_{k}(\theta),
\end{eqnarray*}
where 
\begin{eqnarray}\label{hatF}
\widetilde{F}_{n, \theta}(\xi) 
&=&
 \Big(\frac{1+i}{2}\Big) \widehat{F}_n(\xi) 
 +  \theta(-1)\Big(\frac{1-i}{2}\Big)\widehat{F}_n(-\xi)  \nonumber \\
& = &
\eta_{\theta} \int_{0}^{\infty}F_n(y)\Big(\cos(2\pi \xi y) + \theta(-1) \sin(2\pi \xi y)\Big) dy,
\end{eqnarray}
and
$$
\eta_{\theta}
~=~
\begin{cases}
1&  \text{ if  }~~\theta(-1) = 1 \\
i &  \text{ if  }~~\theta(-1) = -1.
\end{cases}
$$
Putting this in \eqref{P30}, we get that 
\begin{eqnarray}\label{P31}
\mathcal{E}(u, n, \theta) = \frac{1}{2qun}\Big(\frac{16q}{un}\Big)
\overline{\theta (a)}\theta(un)
\!\!\!\sum_{v \le X^{1/8} \atop{ (v, un ) =1}} 
\frac{\mu(v)\theta(v^2)}{v^2}\sum_{k \neq 0} 
\widetilde{F}_{n, \theta}\Big(\frac{kX}{2qunv^2}\Big)G_k(un)
\tau_{k}(\theta).
\end{eqnarray} 
Using Theorem 9.12 (see pages 286-291) of \cite{MV}, we 
have
\begin{eqnarray}\label{gsmv}
\tau_{k}(\theta) 
& = &
\begin{cases}
(-1)^k \theta^*(2)
\overline{\theta^*(k)} 
\tau(\theta^*)
&  \text{ if  }~~(k, 2q)\mid 2\\
0 &  \text{ otherwise}.
\end{cases}
\end{eqnarray}
Putting \eqref{gsmv} in \eqref{P31}, we get
\begin{eqnarray}\label{Newint}
\sum_{n=1}^{\infty} 
\frac{d_\psi(n)}{\sqrt{n}}\mathcal{E}(u, n, \theta)  
&=& 
\frac{\overline{\theta(a)}\theta(u)}{2qu}\Big(\frac{16q}{u}\Big)
\theta^*(2) \tau(\theta^*)
\sum_{v \le X^{1/8} \atop{ (v, u) =1}} 
\frac{\mu(v)\theta(v^2)}{v^2}
\sum_{k \neq 0 } 
(-1)^k \overline{\theta^*(k)} 
\nonumber \\
&&
\sum_{n=1 \atop{ (n, v) =1}}^{\infty}  \frac{d_\psi(n)}{n^{3/2}}\Big(\frac{16q}{n}\Big)
\theta(n)
\widetilde{F}_{n, \theta}\Big(\frac{k X}{2qunv^2}\Big)G_{k}(un).
\end{eqnarray}
We now study the sum 
\begin{eqnarray*}
\sum_{n=1 \atop{ (n, v) =1}}^{\infty}  \frac{d_\psi(n)}{n^{3/2}}\Big(\frac{16q}{n}\Big)
\theta(n)
\widetilde{F}_{n, \theta}\Big(\frac{k X}{2qunv^2}\Big)G_{k}(un).
\end{eqnarray*}
Using Mellin transform (see Chapter 5 of \cite{MV}), we can rewrite the above
sum as 
\begin{eqnarray}\label{Newsum}
\frac{1}{2\pi i}\int_{(c)}
\sum_{n=1 \atop{ (n, v) =1}}^{\infty}  \frac{d_\psi(n)}{n^{3/2 + \omega}}\Big(\frac{16q}{n}\Big)
\theta(n)G_{k}(un)
\Bigg(\int_0^{\infty}\widetilde{F}_{t, \theta}\Big(\frac{k X}{2q u t v^2}\Big) t^{\omega - 1} dt\Bigg) d\omega
\end{eqnarray}
for any $c>1$. As in Section 5.2 of \cite{KS00}, for any $\xi \in \mathbb{R}$ and
$\omega \in \C$ with $\Re(\omega)>0$, we define 
$$
f_{\theta}(\xi , \omega) = \int_0^{\infty}\widetilde{F}_{t, \theta}\Big(\frac{\xi}{t}\Big)t^{\omega - 1}dt.
$$
Applying \eqref{hatF} and arguing as in Lemma 5.2 of \cite{KS00}, we list the 
following properties of $f_{\theta}(\xi ,\omega)$ which will be used subsequently;
\begin{enumerate}
\item
For $\xi \neq 0$ and for any $c$ with $\max(\Re(\omega), 0) < c < 1 + \Re(\omega)$, 
\begin{eqnarray}\label{intf}
f_{\theta}(\xi ,\omega) 
=
\eta_{\theta}\check{\Phi}(\omega)|\xi |^{\omega} \frac{1}{2\pi i}\int_{(c)}
\Big(\frac{\Gamma(\frac{s}{2}+ \frac14)}{\Gamma(\frac14)}\Big)^2
\Big(\frac{8qX}{|\xi |\pi}\Big)^s (2\pi)^{-s + \omega }\Gamma(s-\omega)\nonumber \\
\Big(\cos\big(\frac{\pi}{2}(s - \omega)\big) +  
\theta(-1)\sgn(\xi) \sin\big(\frac{\pi}{2}(s - \omega)\big)\Big)\frac{ds}{s},
\end{eqnarray}
where
$$
\check{\Phi}(\omega) 
=
\int_0^{\infty}\Phi(y)y^{\omega}dy. 
$$

\item
For a fixed $\xi \neq 0$, $f_{\theta}(\xi, \omega)$ is a holomorphic function of $\omega$ in the region 
$\Re(\omega) > -1$. Further, in the region $-1 < \Re(\omega) \le  1$,  it
satisfies the bound 
\begin{equation}\label{bbftheta}
|f_{\theta}(\xi, \omega)| \ll
(1 + |\omega|)^{-\Re(\omega) - 1/2} \exp\Big(-\frac{1}{10}
\frac{\sqrt{|\xi|}}{\sqrt{X(|\omega | + 1)}}\Big)|\xi|^{\omega}\check{\Phi}(\omega).
\end{equation}
\end{enumerate}
Hence \eqref{Newsum} is equal to
\begin{eqnarray}\label{Newsum2}
\frac{1}{2\pi i}\int_{(c)}
\sum_{n=1 \atop{ (n, v) =1}}^{\infty}  \frac{d_\psi(n)}{n^{3/2 + \omega}}
\Big(\frac{16q}{n}\Big)
\theta(n)G_{k}(un)
f_{\theta}\Big(\frac{k X}{2quv^2}, \omega\Big) d\omega .
\end{eqnarray}
Write $k = k_1k_2^2$, where $k_1$ is square-free.
Following the arguments in Lemma 5.3 of \cite{KST} (see also Lemma 5.3 of \cite{KS00}), 
for $\Re(s) >1$, we can write
\begin{eqnarray*}
&&
\sum_{n=1 \atop{ (n, v) =1}}^{\infty}  \frac{d_\psi(n)}{n^{s}}\Big(\frac{16q}{n}\Big)
\frac{\theta(n)G_{k}(un)}{\sqrt{n}}\\
& = & 
L(s , \psi \otimes\theta \otimes \chi_{16q k_1})
L(s , \overline{\psi} \otimes \theta \otimes \chi_{16qk_1})\mathcal{G}(s ,k, u, v,\theta),
\end{eqnarray*}
where
$$
\mathcal{G}(s, k, u, v,\theta) 
= 
\prod_{p }\mathcal{G}_p(s, k, u, v, \theta)
$$ 
and 
\begin{eqnarray*}
\mathcal{G}_p(s, k, u, v ,\theta) 
& = &
\begin{cases}
\displaystyle\Big(1- \frac{\psi(p)\theta(p)}{p^s}\Big(\frac{16q k_1}{p}\Big)\Big)
\Big(1- \frac{\overline{\psi(p)}\theta(p)}{p^s}\Big(\frac{16q k_1}{p}\Big)\Big)
&  \text{ if  }~~p \mid 2qv\\ \\
\displaystyle\Big(1- \frac{\psi(p)\theta(p)}{p^s}\Big(\frac{16q k_1}{p}\Big)\Big)
\Big(1- \frac{\overline{\psi(p)}\theta(p)}{p^s}\Big(\frac{16q k_1}{p}\Big)\Big) \\
\displaystyle\sum_{r=0}^{\infty}\frac{d_\psi(p^{r})\theta(p^{r})}
{p^{r(s + 1/2)}}\Big(\frac{16q }{p^{r}}\Big)
G_{k}(p^{r+b}) &  \text{ if } p\nmid 2 q v, ~p^b || u . \\
\end{cases}
\end{eqnarray*}
Note that when $p\nmid 2q k u v$, then $G_{k}(p) = \Big(\frac{k}{p}\Big)\sqrt{p}=\Big(\frac{k_1}{p}\Big)\sqrt{p}$.
Thus for $p \nmid 2qkuv$, 
$$
\mathcal{G}_p(s, k, u, v, \theta) 
~ = ~
\Bigg(1- \frac{\psi(p)\theta(p)}{p^s}\Big(\frac{16q k_1}{p}\Big)\Bigg)
\Bigg(1- \frac{\overline{\psi(p)}\theta(p)}{p^s}\Big(\frac{16q k_1}{p}\Big)\Bigg)
\Bigg(1 + \frac{d_{\psi}(p)\theta(p)}{p^{s}}\Big(\frac{ 16 q k_1}{p}\Big)\Bigg)
$$
and hence $\mathcal{G}(s, k, u, v,\theta)$ is holomorphic in the region
$\Re(s)>1/2$. Further, for $\Re(s) \ge 1/2 + \epsilon$, it satisfies the bound
\begin{equation}\label{bbg}
|\mathcal{G}(s, k, u, v,\theta)|
\ll_{\epsilon}
(|k| v)^{\epsilon} u^{1/2 + \epsilon} (u, k_2^2)^{1/2}.
\end{equation}
Therefore we can rewrite \eqref{Newsum2} as
\begin{eqnarray*}
\frac{1}{2\pi i}\int_{(c)}
\mathcal{G}(1 + \omega ,k, u, v,\theta)L(1 + \omega , \psi \otimes\theta \otimes \chi_{16q k_1})
L(1 + \omega , \overline{\psi} \otimes \theta \otimes \chi_{16q k_1})
f_{\theta}\Big(\frac{ k X}{2quv^2}, \omega\Big) d\omega .
\end{eqnarray*}
We move the line of integration to $(-1/2 + \epsilon)$
and while doing so, we pick up a pole at $\omega =0$
when either $\psi \otimes\theta \otimes \chi_{16q k_1}$ 
or $\bar{\psi} \otimes\theta \otimes \chi_{16q k_1}$ is a principal
character.  Note that at most one of $\psi \otimes\theta \otimes \chi_{16q k_1}$ 
or $\bar{\psi} \otimes\theta \otimes \chi_{16q k_1}$ is a principal
character as $\psi$ is non-quadratic. If one of these characters is principal,
then for such $\theta$, we denote the unique $k_1$ by $k_{\theta}$. Since the conductors
of $\psi$ and $\theta$ are equal to $q$ and $(k, q) = 1$ (see \eqref{gsmv}), we have 
$k_{\theta} = (-1)^{(q-1)/2}$.
 
Now moving the line of integration to $(-1/2 + \epsilon)$, we see that
\eqref{Newsum2} is equal to
 \begin{eqnarray}\label{intone}
\frac{1}{2\pi i}\int_{(-1/2 + \epsilon)}
\mathcal{G}(1 + \omega, k, u, v,\theta)L(1 + \omega , \psi \otimes\theta \otimes \chi_{16q  k_1})
L(1 + \omega , \overline{\psi} \otimes \theta \otimes \chi_{16qk_1})
f_{\theta}\Big(\frac{ k X}{2quv^2}, \omega\Big) d\omega \nonumber \\
+\phantom{m} \delta_{\theta}(k)  ~~
\mathcal{G}(1, k, u, v,\theta)f_{\theta}\Big(\frac{k X}{2quv^2}, 0 \Big) \mathcal{R}(\theta), 
\phantom{mmmmmmmmmmmm}
\end{eqnarray}
where 
$$
\delta_{\theta }(k) 
= 
\begin{cases}
1 &   \text{ if } k_1 = k_{\theta} \\
0  &  \text{ otherwise} 
\end{cases}
$$
and
$$
\mathcal{R}(\theta) =
{\textrm{res}}_{\omega = 0}\Big(\omega L(1 + \omega, \psi \otimes\theta \otimes 
\chi_{16q k_{\theta}})
L(1 + \omega , \overline{\psi} \otimes \theta \otimes \chi_{16q k_{\theta}})\Big).
$$
The second term in \eqref{intone} will contribute to the main term of the
asymptotic of \eqref{Newint} and we estimate it as follows.
We first consider the sum
\begin{eqnarray}\label{sumK}
&&
\sum_{k \ne 0} (-1)^k \delta_{\theta}(k)  \overline{\theta^*(k)}  
\mathcal{G}(1, k, u, v,\theta)f_{\theta}\Big(\frac{k X}{2quv^2}, 0 \Big) 
\nonumber \\
&=&
\sum_{k_2 \ge 1 }(-1)^{k_2} \overline{\theta^*(k_{\theta}k_2^2)} 
\mathcal{G}(1 ,k_{\theta}k_2^2, u, v,\theta)f_{\theta}
\Big(\frac{k_{\theta}k_2^2 X}{2quv^2}, 0 \Big).
\end{eqnarray}
Using \eqref{intf}, we can write \eqref{sumK} as
\begin{equation*}
\eta_{\theta}\theta^*(k_{\theta})\hat{\Phi}(0)
\frac{1}{2\pi i}\int_{(3/4)}
\Big(\frac{\Gamma(\frac{s}{2}+ \frac14)}{\Gamma(\frac14)}\Big)^2
\Big(\frac{16q^2uv^2}{\pi}\Big)^s
~\Gamma_1(s) 
\sum_{k_2 \ge 1 } (-1)^{k_2}\frac{\overline{\theta^*(k_2^2)} }{k_2^{2s}}
\mathcal{G}(1, k_{\theta}k_2^2, u, v, \theta) \frac{ds}{s},
\end{equation*}
where $\Gamma_1(s) = (2\pi)^{-s } \Gamma(s) \big(\cos(\frac{\pi}{2}s) +  \sin(\frac{\pi}{2}s)\big)$.
Here we use the fact that $\theta(-1)= \sgn(k_{\theta})$. This is true as 
$\psi$ is an even character and $\theta \otimes \chi_{16q k_{\theta}}(-1) = 1$.

For $s, \omega \in \C$ with $\Re(s) > 1/2 ,  |\omega - 1 | <  1/\log X$, we 
now define
$$
\mathcal{H}(s, \omega, u, v, \theta) 
=
 u^{s}  \sum_{k_2 \ge 1 } (-1)^{k_2}\frac{\overline{\theta^*(k_2^2)} }{k_2^{2s}}
\mathcal{G}(\omega, k_{\theta}k_2^2, u, v, \theta).
$$
Note that the above series converges absolutely for 
$s, \omega \in \C$ with $\Re(s) > 1/2 ,  |\omega - 1 | <  1/\log X$.
Using the fact that 
$$
\mathcal{G}(\omega, 4k_{\theta}k_2^2, u, v,\theta) 
= \mathcal{G}(\omega, k_{\theta} k_2^2,u, v,\theta),
$$ 
we get
\begin{eqnarray*}
\sum_{k_2 \ge 1 } (-1)^{k_2}
\frac{\mathcal{G}(\omega, k_{\theta}k_2^2, u, v,\theta)}
{k_2^{2s}}\overline{\theta^*(k_2^2)}  
~=~ 
-\Big(1 - \frac{2\overline{\theta^*(2)}^2}{2^{2s}}\Big)\sum_{k_2 \ge 1 } 
\frac{\mathcal{G}(\omega, k_{\theta }k_2^2, u, v,\theta)}
{k_2^{2s}}\overline{\theta^*(k_2)}^2.
\end{eqnarray*}
Further, we have
\begin{eqnarray*}
\sum_{k_2 \ge 1}  \frac{\mathcal{G}(\omega, 
k_{\theta} k_2^2, u, v,\theta)}{k_2^{2s}}
\overline{\theta^*(k_2)}^2 
&=&
\prod_p \sum_{b \ge 0 }  \frac{
\mathcal{G}_p(\omega, k_{\theta} p^{2b}, u, v,\theta)}{p^{2bs}}
\overline{\theta^*(p^{2b})}.
\end{eqnarray*}
Here we have used the fact that 
$$
\mathcal{G}_p(\omega, k_{\theta} p^{2b}k_3^2, u, v,\theta) 
= \mathcal{G}_p(\omega, k_{\theta} p^{2b}, u, v,\theta)
$$ 
for any positive integer $k_3$ with $(p, k_3)= 1$. 
Therefore
\begin{eqnarray*}
&&
\mathcal{H}(s, \omega, u, v, \theta) 
=
 - u^{s} \Bigg( 1 - \frac{2\overline{\theta^*(2)}^2}{2^{2s}} \Bigg)
\prod_p
 \sum_{b \ge 0 }  \frac{\mathcal{G}_p(\omega, k_{\theta } p^{2b}, u, v, \theta)}{p^{2bs}}
\overline{\theta^*(p^{2b})} .
\end{eqnarray*}
Using the equalities
\begin{eqnarray}\label{Ana2}
\sum_{b = 0}^{\infty}
\frac{d_{\psi}(p^{2b})}{p^{(2s + 2\omega -1)b}} 
& = & \Bigg( 1 + \frac{1}{p^{2s+2\omega-1}}\Bigg)
\Bigg(1 - \frac{\psi(p)^2}{p^{2s+2\omega -1}}\Bigg)^{-1}
\Bigg(1 - \frac{\overline{\psi(p)}^2}{p^{2s+2\omega -1}}\Bigg)^{-1}
\end{eqnarray}
and
\begin{eqnarray}\label{Ana3}
\sum_{b = 0}^{\infty}
\frac{d_{\psi}(p^{2b+1})}{p^{(2s + 2\omega -1)b}} 
& = & 
d_{\psi}(p)
\Bigg(1 - \frac{\psi(p)^2}{p^{2s+2\omega -1}}\Bigg)^{-1}
\Bigg(1 - \frac{\overline{\psi(p)}^2}{p^{2s+2\omega -1}}\Bigg)^{-1},
\end{eqnarray}
we get
\begin{eqnarray*}
\mathcal{H}(s, \omega, u, v, \theta)
& = &
- u u_1^{s - 1/2} \overline{\theta(u)}\Big(\frac{16q}{u}\Big) 
\Bigg( 1 - \frac{2\overline{\theta^*(2)}^2}{2^{2s}}
\Bigg)L(2s, \overline{\theta^*}^2)
 \prod_p \mathcal{H}_{1, p}(s, \omega, u, v, \theta).
\end{eqnarray*}
Here
\begin{eqnarray*}
&&
\mathcal{H}_{1, p}(s, \omega, u, v, \theta)\\ \\
&&\!\!\!\!\!\!\!\!\!\!\!\! =  
\begin{cases}
\Big(1- \frac{\psi(p)\theta(p)}{p^\omega}\Big(\frac{16q k_{\theta}}{p}\Big)\Big)
\Big(1- \frac{\overline{\psi(p)}\theta(p)}{p^\omega}\Big(\frac{16q k_{\theta}}{p}\Big)\Big)
& \text{if } p \mid 2qv \\ \\
\Big(1- \frac{\psi(p)\theta(p)}{p^\omega}\Big(\frac{16q k_{\theta}}{p}\Big)\Big)
\Big(1- \frac{\overline{\psi(p)}\theta(p)}{p^\omega}\Big(\frac{16q k_{\theta}}{p}\Big)\Big)
\Big(1 - \frac{\psi(p)^2}{p^{2s + 2\omega -1}}\Big)^{-1}
\Big(1 - \frac{\overline{\psi(p)}^2}{p^{2s +2\omega -1}}\Big)^{-1} 
\\
\Big\{ 1  + \frac{1}{p^{2s+2\omega-1}}
  - \frac{(d_{\psi}(p)^2 -1)}{p^{2s+2\omega}}
 + \frac{1}{p^{4s+4\omega - 1}}
+ \frac{d_{\psi}(p)\theta(p)}{p^\omega}
\Big(\frac{16q  k_{\theta}}{p}\Big)
-\frac{d_{\psi}(p)\overline{\theta(p)}}{p^{\omega + 2s}}
\Big(\frac{16q k_{\theta}}{p}\Big)
\Big\}
& \text{if } p \nmid 2quv 
\\ \\
\Big(1- \frac{\psi(p)\theta(p)}{p^\omega}\Big(\frac{16q k_{\theta}}{p}\Big)\Big)
\Big(1- \frac{\overline{\psi(p)}\theta(p)}{p^\omega}\Big(\frac{16q k_{\theta}}{p}\Big)\Big)
\Big(1 - \frac{\psi(p)^2}{p^{2s+2\omega -1}}\Big)^{-1}
\Big(1 - \frac{\overline{\psi(p)}^2}{p^{2s+2\omega -1}}\Big)^{-1}
\\
\Big\{
\frac{d_{\psi}(p)}{p^{\omega + 2s -1}} - \frac{d_{\psi}(p)}{p^{\omega + 2s }} + 
\theta(p)\Big(\frac{16 q k_{\theta }}{p}\Big)
\Bigg(1 + \frac{1}{p^{2s+2\omega-1}}
 - \frac{\overline{\theta(p^2)}}{p^{2s}}
- \frac{\overline{\theta(p^2)}}{p^{2\omega + 4s -1}}\Bigg)
\Big\}
& \text{if } p \mid u_1
\\ \\
\Big(1- \frac{\psi(p)\theta(p)}{p^\omega}\Big(\frac{16q k_{\theta}}{p}\Big)\Big)
\Big(1- \frac{\overline{\psi(p)}\theta(p)}{p^\omega}\Big(\frac{16q k_{\theta}}{p}\Big)\Big)
\Big(1 - \frac{\psi(p)^2}{p^{2s+2\omega -1}}\Big)^{-1}
\Big(1 - \frac{\overline{\psi(p)}^2}{p^{2s+2\omega -1}}\Big)^{-1}
\\
\Big\{
1-\frac{1}{p} + \frac{1}{p^{2s+2\omega-1}} -
\frac{1}{p^{2s + 2\omega}} 
+ \theta(p)\Big(\frac{16q k_{\theta }}{p}\Big)\frac{d_{\psi}(p)}{p^\omega}
 - \overline{\theta(p)} \Big(\frac{16q k_{\theta} }{p}\Big)
 \frac{d_{\psi}(p)}{p^{\omega + 2s}}
\Big\} & \text{if } p \mid u_2.
\end{cases}
\end{eqnarray*}
From this it follows that $\mathcal{H}(s, 1, u,  v, \theta)$ is a holomorphic function
when $\Re(s)>0$ and hence we can write \eqref{sumK} as
\begin{eqnarray}\label{moveint}
&&
\sum_{k_2 \ge 1 } (-1)^{k_2} \overline{\theta^*(k_{\theta}k_2^2)} 
\mathcal{G}(1 ,k_{\theta}k_2^2, u, v,\theta)f_{\theta}
\Big(\frac{k_{\theta }k_2^2 X}{2q u v^2}, 0 \Big) \\
& = &
\eta_{\theta}\theta^*(k_{\theta})\hat{\Phi}(0)
\frac{1}{2\pi i}\int_{(3/4)}
\Big(\frac{\Gamma(\frac{s}{2}+ \frac14)}{\Gamma(\frac14)}\Big)^2
\Big(\frac{16q^2v^2}{\pi}\Big)^{s} \Gamma_1(s)
\mathcal{H}(s, 1, u, v, \theta) \frac{ds}{s}. \nonumber
\end{eqnarray}
Using Rademacher principle, for $\Re(s) \ge 1/{\log X}$, we have 
$L(2s, \overline{\theta^*}^2) \ll (1+ |s|)$. Applying this,  for $\Re(s) \ge 1/{\log X}$,
we have
\begin{eqnarray}\label{2026}
|\mathcal{H}(s, 1, u, v, \theta)|
& \ll & 
u^{1 + \epsilon} u_1^{\Re(s)-1/2}(vX)^{\epsilon}(1 + |s|).
\end{eqnarray}
Hence we can move the line of integration in \eqref{moveint} to $\Re(s) = 1/{\log X}$. Thus
the contribution of the second term in \eqref{intone}
to \eqref{Newint} can be written as
\begin{eqnarray*}
&&
\eta_{\theta}\hat{\Phi}(0)
\frac{\overline{\theta(a)}\theta(u)}{2qu}
\Big(\frac{16 q}{u}\Big)\theta^*(2k_{\theta})
\tau(\theta^*) \mathcal{R}(\theta) \sum_{v \le X^{1/8} \atop{ (v, u ) =1}} 
\frac{\mu(v)\theta(v^2)}{v^{2}} \\
&&
\frac{1}{2\pi i}\int_{(\frac{1}{\log X})}
\Big(\frac{\Gamma(\frac{s}{2}+ \frac14)}{\Gamma(\frac14)}\Big)^2
\Big(\frac{16 q^2v^2}{\pi}\Big)^{s} \Gamma_1(s)\mathcal{H}(s, 1, u, v, \theta) \frac{ds}{s}.
\end{eqnarray*}
Interchanging the sum and the integral, we get
\begin{eqnarray}\label{Residue1}
&&
\eta_{\theta}\hat{\Phi}(0)
\frac{\overline{\theta(a)}\theta(u)}{2qu}\Big(\frac{16 q}{u}\Big)
\theta^*(2k_{\theta})\tau(\theta^*) \mathcal{R}(\theta)
\frac{1}{2\pi i}\int_{(\frac{1}{\log X})}
\Big(\frac{\Gamma(\frac{s}{2}+ \frac14)}{\Gamma(\frac14)}\Big)^2
\Big(\frac{16 q^2}{\pi}\Big)^{s} \Gamma_1(s)
\nonumber \\
&& \phantom{mmm}
\sum_{v \le X^{1/8} \atop{ (v, u ) =1}} 
\frac{\mu(v)\theta(v^2)}{v^{2-2s}}
\mathcal{H}(s, 1, u, v, \theta) \frac{ds}{s}.
\end{eqnarray}
Extending the sum over $v$ to infinity and using \eqref{2026},
we see that the contribution of the error term is 
$~~\displaystyle{\ll~\frac{(uX)^{\epsilon}}{ u_1^{1/2} X^{1/8} } }$.
When $\theta = \Psi\otimes \chi_{\pm q}\otimes \chi_0$,
where $\Psi = \psi$ or $\bar{\psi}$,
we can write $\mathcal{H}(s, 1, u, v, \theta) $ as
\begin{eqnarray*}
- u \overline{\Psi(u)}
\Big(\frac{k_{\theta} }{u}\Big)
u_1^{s - 1/2} \Psi(u_1) 
\Bigg( 1 - \frac{2\overline{\Psi(2)}^2}{2^{2s}}
\Bigg)L(2s, \overline{\Psi}^2)
L(2s + 1, \Psi^2)
 \prod_p \mathcal{H}_{2, p}(s, 1, u, v, \Psi),
\end{eqnarray*}
where
\begin{eqnarray*}
\mathcal{H}_{2, p}(s, 1, u, v, \Psi)
& = &  
\begin{cases}
\Big(1 - \frac{\Psi(p)^2}{p^{2s +1}}\Big) & p|2q \\
\Big(1 - \frac{1}{p}\Big)
\Big(1 - \frac{\Psi(p)^2}{p}\Big)
\Big(1 - \frac{\Psi(p)^2}{p^{2s +1}}\Big)
 & p \mid v , p\nmid 2q\\
\Big(1 - \frac{1}{p}\Big)
\Big(1 - \frac{\Psi(p)^2}{p}\Big)
\Big( 1 + \frac{\Psi(p)^2}{p} + \frac{1}{p} - \frac{\Psi(p)^2}{p^{2s + 2}} \Big)
& p \nmid 2quv 
\\
\Big(1 - \frac{1}{p}\Big)
\Big(1 - \frac{\Psi(p)^2}{p}\Big)
\Big( 1 + \frac{1}{p^{2s}} \Big)
& p \mid u_1
\\
\Big(1 - \frac{1}{p}\Big)
\Big(1 - \frac{\Psi(p)^2}{p}\Big)
\Big( 1 + \frac{\Psi(p)^2}{p}\Big)
 & p \mid u_2.
\end{cases}
\end{eqnarray*}
Then we have 
\begin{eqnarray}\label{lab1}
&&
\sum_{v \atop{ (v, 2qu) =1}} 
\frac{\mu(v)\Psi(v^2) }{v^{2-2s}}
\mathcal{H}(s, 1, u, v, \Psi \otimes \chi_{\pm q} \otimes \chi_0) \nonumber \\
& = &
- u \overline{\Psi(u)}\Big(\frac{k_{\theta}}{u}\Big)u_1^{s - 1/2}
 \Psi(u_1)
\Big( 1 - \frac{2\overline{\Psi(2)^2}}{2^{2s}}\Big)
\Big( 1- \frac{\Psi(2)^2}{2^{2s + 1}}\Big)
L(2s, \overline{\Psi^2}) 
L( 2s + 1,  \Psi^2) \nonumber \\
&&
\prod_{p|u_1}
 \Big(1 - \frac{1}{p}\Big)
\Big(1 - \frac{\Psi(p)^2}{p}\Big)
\Big( 1 + \frac{1}{p^{2s}} \Big)
\prod_{p|u_2 \atop{p \nmid u_1}}
 \Big(1 - \frac{1}{p}\Big)
\Big(1 - \frac{\Psi(p)^2}{p}\Big)
\Big( 1 + \frac{\Psi(p)^2}{p}\Big) \nonumber \\
&& 
\prod_{p \nmid 2qu}
 \Big(1 - \frac{1}{p}\Big)
\Big(1 - \frac{\Psi(p)^2}{p}\Big)
\Big( 1 + \frac{\Psi(p)^2}{p} + \frac{1}{p} - \frac{\Psi(p)^2}{p^{2s + 2}} 
- \frac{\Psi(p)^2}{p^{2-2s}} + \frac{\Psi(p)^4}{p^3} \Big).
\end{eqnarray}
Recall that 
$$
\mathcal{R}(\theta)
=
{\textrm{res}}_{\omega = 0}\Big(\omega L(1+ \omega, \psi \otimes\theta \otimes \chi_{16qk_{\theta}})
L(1 + \omega , \overline{\psi} \otimes \theta \otimes \chi_{16q k_{\theta}})\Big).
$$
Therefore
\begin{equation}\label{lab2}
\mathcal{R}(\Psi \otimes \chi_{\pm q} \otimes \chi_0) 
=
 \frac{\varphi(q)}{2q}\Big(1-\frac{\Psi(2)^2}{2}\Big)L(1, \Psi^2).
\end{equation}
Note that $\theta^* = \Psi \otimes \chi_{\pm q}$. Since $\psi$ is an
even character, $\eta_{\theta} = \eta_q$.
Using \eqref{lab1} and \eqref{lab2}, we can write \eqref{Residue1} as
\begin{eqnarray}\label{Resterm1}
&&
(-1)^{\frac{q+1}{2}} \eta_{q}\hat{\Phi}(0)
\frac{\overline{\Psi(a)}}{2q}\Big(\frac{ \pm q}{2a}\Big)
\frac{\Psi(2u_1)}{\sqrt{u_1}}
 \tau(\Psi \otimes \chi_{\pm q}) 
\mathcal{R}(\Psi \otimes \chi_{\pm q} \otimes \chi_0) 
\frac{1}{2\pi i}\int_{(\frac{1}{\log X})}
\Big(\frac{\Gamma(\frac{s}{2}+ \frac14)}{\Gamma(\frac14)}\Big)^2
\nonumber \\
&&
\Big(\frac{16 q^2}{\pi}\Big)^{s}
\Big( 1 - \frac{2\overline{\Psi(2)^2}}{2^{2s}}\Big)
\Big( 1- \frac{\Psi(2)^2}{2^{2s + 1}}\Big)
\Gamma_1(s)
L(2s, \overline{\Psi^2}) L( 2s + 1, \Psi^2)
\nonumber \\
&&
\prod_{p|u_1}  \Big(1 - \frac{1}{p}\Big)
\Big(1 - \frac{\Psi(p)^2}{p}\Big)
\Big( p^{s } + \frac{1}{p^{s}} \Big)
\prod_{p|u_2 \atop{p \nmid u_1}}
 \Big(1 - \frac{1}{p}\Big)
\Big(1 - \frac{\Psi(p)^2}{p}\Big)
\Big( 1 + \frac{\Psi(p)^2}{p}\Big)
\nonumber \\
&&
\prod_{p \nmid 2qu}
 \Big(1 - \frac{1}{p}\Big)
\Big(1 - \frac{\Psi(p)^2}{p}\Big)
\Big( 1 + \frac{\Psi(p)^2}{p} + \frac{1}{p} + \frac{\Psi(p)^4}{p^3} 
- \frac{\Psi(p)^2}{p^{2s + 2}}  - \frac{\Psi(p)^2}{p^{2-2s}}  \Big)
 \frac{ds}{s}.  \nonumber
\end{eqnarray}
We denote the above term by $\mathcal{M}_2(\Psi)$. We now shift the line of 
integration to $(- 1/\log X )$. Let $\mathcal{R}_1(\Psi)$ be the residue of the
simple pole at $s=0$. It is easy to see that
\begin{eqnarray*}
\mathcal{R}_1(\Psi) 
&=&
(-1)^{\frac{q+1}{2}} \eta_{q}\hat{\Phi}(0)\overline{\Psi(a)}
\frac{\varphi(q)}{2q^2}\Big(\frac{ \pm q}{2a}\Big)
\frac{\Psi(2u_1)}{\sqrt{u_1}}d(u_1) L'(0, \overline{\Psi}^2) L(1, \Psi^2)^2
 \tau(\Psi \otimes \chi_{\pm q})  \\
&&
\Big(1-\frac{\psi(2)^2}{2}\Big)^2 \Big(1 - 2\overline{\psi(2)}^2 \Big)
\Big)
\prod_{p|u}
\Big(1 - \frac{1}{p}\Big) \Big(1 - \frac{\Psi(p)^2}{p}\Big)
\prod_{p|u_2 \atop{p \nmid u_1}}
\Big( 1 + \frac{\Psi(p)^2}{p}\Big)  
\\
&&
\prod_{p \nmid 2qu} \Big(1 - \frac{1}{p}\Big) \Big(1 - \frac{\Psi(p)^2}{p}\Big) 
\Big(1 + \frac{\Psi(p)^2}{p} + \frac{1}{p} + \frac{\Psi(p)^4}{p^3}
- 2\frac{\Psi(p)^2}{p^{2}}
\Big).
\end{eqnarray*}
Write
\begin{eqnarray}\label{Resterm4}
\mathcal{J}(s, \Psi , u)
&=&
\Big(\frac{\Gamma(\frac{s}{2}+ \frac14)}{\Gamma(\frac14)}\Big)^2
\Big(\frac{4 q^2}{\pi}\Big)^{s} \Big( 2^{2s} - 2\overline{\Psi(2)^2 }
- \frac{\Psi(2)^2}{2} + \frac{1}{2^{2s}}\Big)
\Gamma_1(s)
L(2s, \overline{\Psi^2}) L( 2s + 1, \Psi^2)
\nonumber \\
&&
\prod_{p|u_1}  \Big(1 - \frac{1}{p}\Big)
\Big(1 - \frac{\Psi(p)^2}{p}\Big)
\Big( p^{s } + \frac{1}{p^{s}} \Big)
\prod_{p|u_2 \atop{p \nmid u_1}}
 \Big(1 - \frac{1}{p}\Big)
\Big(1 - \frac{\Psi(p)^2}{p}\Big)
\Big( 1 + \frac{\Psi(p)^2}{p}\Big)
\nonumber \\
&&
\prod_{p \nmid 2qu}
 \Big(1 - \frac{1}{p}\Big)
\Big(1 - \frac{\Psi(p)^2}{p}\Big)
\Big( 1 + \frac{\Psi(p)^2}{p} + \frac{1}{p} - \frac{\Psi(p)^2}{p^{2s + 2}} 
- \frac{\Psi(p)^2}{p^{2-2s}} + \frac{\Psi(p)^4}{p^3} \Big).
\end{eqnarray}
Using properties of the $\Gamma$-function and the functional equation for $L(s, \Psi^2)$,
we note that 
$$
\mathcal{J}(s, \Psi , u) = \mathcal{J}(-s, \Psi , u).
$$
Hence 
$$
2\mathcal{M}_2(\Psi) = \mathcal{R}_1 (\Psi).
$$
Thus we have
\begin{eqnarray}\label{lab3}
&&
\frac{\mathcal{R}_1(\psi) + \mathcal{R}_1 (\bar\psi)}{2}  \\
&=&
\hat{\Phi}(0)\Bigg(C_{\psi}\overline{\psi(a)}\Big(\frac{\pm q}{a}\Big)
\frac{d(u_1)}{h_{\psi}(u)}
\frac{ \psi(u_1) \sqrt{u_1}}{ \prod_{p| u_1}( p + \psi(p)^2)}
~+~
\overline{C_{\psi}}\psi(a)\Big(\frac{\pm q}{a}\Big)
 \frac{d(u_1)}{\overline{h_{\psi}(u)}}
\frac{\overline{\psi(u_1)} \sqrt{u_1}}{\prod_{p| u_1}(p + \overline{\psi(p)}^2)}
\Bigg), \nonumber
\end{eqnarray} 
where $C_{\psi} , h_{\psi}$  are as in the statement. Here we have used the fact that
$C_{\bar{\psi}} = \overline{C_{\psi}}$.
Finally from \eqref{Newint}, \eqref{intone} and using \eqref{bbftheta}, \eqref{bbg}
and Cauchy-Schwarz inequality, we have
\begin{eqnarray}\label{lastEst1}
&&
\frac{\overline{\theta (a)}\theta(u)}{2qu}\Big(\frac{16q}{u}\Big)
\theta^*(2)\tau(\theta^*)
\sum_{v \le X^{1/8} \atop{ (v, u ) =1}} 
\frac{\mu(v)\theta(v^2)}{v^2}
\sum_{k \neq 0} (-1)^k \overline{\theta^*(k)}  \nonumber \\
&&
\frac{1}{2\pi i}\int_{(-1/2 + \epsilon)}
\mathcal{G}(1 + \omega, k, u, v,\theta)L(1 + \omega , \psi \otimes\theta \otimes \chi_{16qk_1})
L(1 + \omega , \overline{\psi} \otimes \theta \otimes \chi_{16qk_1})
f_{\theta}\Big(\frac{kX}{2quv^2}, \omega\Big) d\omega \nonumber \\
&\ll&
\frac{u^{1/2 + \epsilon}}{X^{1/2 - \epsilon}}
\sum_{v \le X^{1/8} \atop{ (v, u ) =1}} 
v^{\epsilon}
\int_{(-1/2 + \epsilon)}
 \check{\Phi}(\omega) \sum_{|k_1| \neq 0}
\frac{(|\omega | + 1)^{1/2 + \epsilon}}{|k_1|} 
\exp\Big(-\frac{1}{20}\frac{\sqrt{|k_1|}}{\sqrt{2quv^2(|\omega | + 1)}}\Big) \nonumber \\
 && 
\Bigg( \Big| L(1 + \omega , \psi \otimes\theta \otimes \chi_{16 qk_1})\Big|^2 + \Big|
L(1 + \omega , \overline{\psi} \otimes \theta \otimes \chi_{16 qk_1})\Big|^2 \Bigg) 
|d\omega|.
\end{eqnarray}
Splitting the sum over $k_1$ into dyadic blocks and using Lemma 2.5 of \cite{KS00},
we see that \eqref{lastEst1} is bounded by 
$$
\ll 
\frac{u^{1/2 + \epsilon}}{X^{3/8 - \epsilon}}.
$$ 
Hence 
\begin{eqnarray}\label{Prin2}
&&
\frac{X}{\varphi(q)} \sum_{\theta \atop{ \psi\otimes\theta \text{ or }  \bar{\psi}\otimes\theta\text { is quadratic }}}
\sum_{n=1}^{\infty} 
\frac{d_\psi(n)}{\sqrt{n}} \mathcal{E}(u,n, \theta)  \nonumber \\
&=&
\frac{\hat{\Phi}(0)X}{\varphi(q)}\Bigg(C_{\psi}\overline{\psi(a)}\Big(\frac{\pm q}{a}\Big)
\frac{d(u_1)}{h_{\psi}(u)}
\frac{ \psi(u_1) \sqrt{u_1}}{ \prod_{p| u_1}( p + \psi(p)^2)} 
~+~
\overline{C_{\psi}}\psi(a)\Big(\frac{\pm q}{a}\Big) \frac{d(u_1)}{\overline{h_{\psi}(u)}}
\frac{\overline{\psi(u_1)} \sqrt{u_1}}{\prod_{p| u_1}(p + \overline{\psi(p)}^2)} \Bigg) \nonumber \\
&&+~~~ O\Big(u^{1/2 + \epsilon}X^{5/8 + \epsilon} + X^{7/8 + \epsilon} \Big). 
\end{eqnarray}
From \eqref{split}, \eqref{last1}, \eqref{MandE}, \eqref{finmtest}, \eqref{finerrorest} and \eqref{Prin2}, we get the result. 
\end{proof}
We now use \lemref{Mainterm} to prove \propref{M2}.
\begin{proof}[Proof of \propref{M2}]
We have 
\begin{eqnarray}\label{nt1}
&&
\sideset{}{'}\sum_{d \atop d \equiv a \!\!\!\!\!\pmod{q}} |L(1/2,\poc)|^2 |R(d)|^2
\Phi\Big(\frac{d}{X}\Big) \nonumber \\
&=&
\sum_{m,n \leq N}r(m)r(n)\psi_0(m)\overline{\psi_0(n)}
\sideset{}{'}\sum_{d \atop d \equiv a \!\!\!\!\!\pmod{q}} \cchi(mn)|L(1/2,\poc)|^2
\Phi\Big(\frac{d}{X}\Big).
\end{eqnarray}
Using \lemref{Mainterm}, we see that the contribution of the error term
from \eqref{nt} is
\begin{align*}
&\ll 
\sum_{m,n \leq N} (X^{7/8+\epsilon}
+ (mn)^{1/2+\epsilon}X^{3/4+\epsilon})
~\ll~
N^2X^{7/8+\epsilon} + N^{3+\epsilon}X^{3/4+\epsilon}
~\ll~
X^{23/24+\epsilon}.
\end{align*}
Thus applying \eqref{nt} of \lemref{Mainterm}  
in \eqref{nt1}, we have

\begin{eqnarray}\label{Modified}
&&
\sideset{}{'}\sum_{d \atop d \equiv a \!\!\!\!\!\pmod{q}} |L(1/2,\poc)|^2 |R(d)|^2
\Phi\Big(\frac{d}{X}\Big)\nonumber\\
&=&
\frac{A_{\psi}}{\varphi(q)}\hat{\Phi}(0)X
\sum_{m,n \le N} \Bigg\{r(m) r(n)\psi_0(m)\overline{\psi_0(n)}
~\frac{\dsi\big(\frac{mn}{(m,n)^2}\big)}{h(mn)}
\frac{\sqrt{\frac{mn}{(m,n)^2}}}{\sigma\big(\frac{mn}{(m,n)^2}\big)} \nonumber \\
 &&  \qquad \Bigg(\log (D_1X) 
 ~+~
\sum_{p|\frac{mn}{(m,n)^2}}\Big(\frac{D_2(p)}{p}-1 \Big)\log p
~+~
 \sum_{p|(m,n)}\frac{D_3(p)}{p}\log p \Bigg)\Bigg\} \nonumber \\
 &+&
 \frac{2\hat{\Phi}(0)X}{\varphi(q)}C_{\psi}\overline{\psi(a)}\Big(\frac{\pm q}{a}\Big)
\sum_{m,n \le N} \Bigg( r(m) r(n)\psi_0(m)\overline{\psi_0(n)}
\frac{d\big(\frac{mn}{(m,n)^2}\big)}{h_{\psi}(mn)}
\frac{ \psi\big(\frac{mn}{(m,n)^2}\big) \sqrt{\frac{mn}{(m,n)^2}}}{ \prod_{p|\frac{mn}{(m,n)^2}}( p + \psi(p)^2)}\Bigg)
\nonumber \\
&+&
 \frac{2\hat{\Phi}(0)X}{\varphi(q)}\overline{C_{\psi}}\psi(a)\Big(\frac{\pm q}{a}\Big) 
 \sum_{m,n \le N} \Bigg( r(m) r(n)\psi_0(m)\overline{\psi_0(n)}
 \frac{d\big(\frac{mn}{(m,n)^2}\big)}{\overline{h_{\psi}(mn)}}
\frac{\overline{\psi\big(\frac{mn}{(m,n)^2}\big)} 
\sqrt{\frac{mn}{(m,n)^2}}}{\prod_{p|\frac{mn}{(m,n)^2}}(p + \overline{\psi(p)}^2)}\Bigg)  \nonumber \\
 &+& O(X^{23/24+\epsilon}).
\end{eqnarray}
We first estimate the sums over all $m,n \ge 1$ and then show that the contribution from the 
terms with $\max\{m, n \} >N$ 
is negligible. Throughout the proof we will use the fact that the function
$r$ takes the value zero on integers which are not square-free. We also use the multiplicative 
property of other functions 
appearing in the sum without specifying at each instance. 
We set
\begin{align*}
S_1 
&=
\sum_{m,n \ge 1} \Bigg\{r(m) r(n)\psi_0(m)\overline{\psi_0(n)}
~\frac{\dsi\big(\frac{mn}{(m,n)^2}\big)}{h(mn)}
\frac{\sqrt{\frac{mn}{(m,n)^2}}}{\sigma\big(\frac{mn}{(m,n)^2}\big)} \\
 &\qquad
 \Bigg(\log (D_1X) 
 ~+~
\sum_{p|\frac{mn}{(m,n)^2}}\Big(\frac{D_2(p)}{p}-1 \Big)\log p
~+~
 \sum_{p|(m,n)}\frac{D_3(p)}{p}\log p \Bigg)\Bigg\}.
\end{align*}
Then we have
\begin{align*}
S_1  
& = 
\sum_{t \ge 1}\sum_{m,n \ge 1 \atop{(m,n)=1}}
\Bigg\{r(mt) r(nt)\psi_0(mt) \overline{\psi_0(nt)}
~\frac{\dsi(mn)}{h(mnt^2)}
\frac{\sqrt{mn}}{\sigma(mn)} \\
 &  \qquad \Bigg(\log (D_1X) 
 ~+~
\sum_{p|mn}\Big(\frac{D_2(p)}{p}-1 \Big)\log p
~+~
 \sum_{p|t}\frac{D_3(p)}{p}\log p \Bigg)\Bigg\}\\
  & = 
\sum_{t \ge 1}\frac{1}{h(t^2)}\sum_{m,n \ge 1 \atop{(m,n)=1}}
\Bigg\{r(mt) r(nt)\psi_0(mt) \overline{\psi_0(nt)}
~\frac{\dsi(m)\dsi(n)}{h(m)h(n)}
\frac{\sqrt{mn}}{\sigma(m)\sigma(n)} \\
 &  \qquad \Bigg(\log (D_1X) 
 ~+~
\sum_{p|m}\Big(\frac{D_2(p)}{p}-1 \Big)\log p
 ~+~
\sum_{p|n}\Big(\frac{D_2(p)}{p}-1 \Big)\log p
~+~
 \sum_{p|t}\frac{D_3(p)}{p}\log p \Bigg)\Bigg\}.
\end{align*}
Using the fact that $\displaystyle\sum_{\beta|b}\mu(\beta)=0$ unless $b = 1$, we get
\begin{align*}
S_1 & = 
\sum_{t \ge 1}\frac{1}{h(t^2)}\sum_{\beta \ge 1}\mu(\beta)\sum_{m,n \ge 1 }
\Bigg\{r(m\beta t)r(n\beta t)\psi_0(m\beta t) \overline{\psi_0(n \beta t)}
~\frac{\dsi(m\beta)\dsi(n\beta)}{h(m\beta)h(n\beta)}
\frac{\sqrt{mn}\beta}{\sigma(m\beta)\sigma(n\beta)}\\
&~ \qquad \Bigg(\log (D_1X) + \sum_{p|m\beta}\Big(\frac{D_2(p)}{p}-1 \Big)\log p
+  \sum_{p|n\beta}\Big(\frac{D_2(p)}{p}-1 \Big)\log p
+ \sum_{p|t}\frac{D_3(p)}{p}\log p \Bigg)\Bigg\}\\
& = 
\sum_{t \ge 1}\frac{1}{h(t^2)}
\sum_{\beta \ge 1}\mu(\beta)\frac{\dsi(\beta)^2\beta}{h(\beta)^2\sigma(\beta)^2}
\sum_{m,n \ge 1 }
\Bigg\{r(m\beta t)r(n\beta t)\psi_0(m\beta t)\overline{\psi_0(n \beta t)}
~\frac{\dsi(m)\dsi(n)}{h(m)h(n)}
\frac{\sqrt{mn}}{\sigma(m)\sigma(n)} \\ 
& \qquad \Bigg(\log (D_1X) + \sum_{p|m}\Big(\frac{D_2(p)}{p}-1 \Big)\log p
+ \sum_{p|n}\Big(\frac{D_2(p)}{p}-1 \Big)\log p \\
&  \qquad \qquad
+~ 2 \sum_{p|\beta}\Big(\frac{D_2(p)}{p}-1 \Big)\log p
+ \sum_{p|t}\frac{D_3(p)}{p}\log p \Bigg)\Bigg\}.
\end{align*}
Writing $\gamma = \beta t$ and using the fact that $h(p^k)=h(p)$ for all primes 
$p$ and integers $k \ge 1$, we have
\begin{align}\label{EstS_1}
S_1 & ~=
\sum_{\gamma \ge 1} \frac{1}{h(\gamma)}
\sum_{\beta | \gamma}\frac{\mu(\beta)\dsi(\beta)^2\beta}{h(\beta)\sigma(\beta)^2}
\sum_{m,n \ge 1 }
r(m\gamma)r(n\gamma)\psi_0(m\gamma)
\overline{\psi_0(n\gamma)}
~\frac{\dsi(m)\dsi(n)}{h(m)h(n)}
\frac{\sqrt{mn}}{\sigma(m)\sigma(n)}\nonumber\\
&~ \qquad \Bigg(\log (D_1X) + \sum_{p|m}\Big(\frac{D_2(p)}{p}-1 \Big)\log p
+ \sum_{p|n}\Big(\frac{D_2(p)}{p}-1 \Big)\log p \nonumber\\
& \qquad \qquad +2\sum_{p|\beta}\Big(\frac{D_2(p)}{p}-1 \Big)\log p
+ \sum_{p|\frac{\gamma}{\beta}}\frac{D_3(p)}{p}\log p \Bigg)\nonumber\\
& = M ~+~ E_1 ~+~ E_2,  \text{ say}.
\end{align}
Here $M$ consists of terms involving $\log (D_1X) $, $E_1$ consists of 
the terms involving $m, n$ and $E_2$ consists of the terms involving $\beta, \gamma/\beta$ 
in the innermost sum. We first estimate $M$ which will contribute to the main term.
Note that
$$
M =
\log (D_1X) \sum_{\gamma \ge 1} \frac{1}{h(\gamma)}
\sum_{\beta | \gamma}\frac{\mu(\beta)\dsi(\beta)^2\beta}{h(\beta)\sigma(\beta)^2}
\sum_{m,n \ge 1}
r(m\gamma)r(n\gamma)\psi_0(m\gamma)\overline{\psi_0(n\gamma)}
~\frac{\dsi(m)\dsi(n)}{h(m)h(n)}
\frac{\sqrt{mn}}{\sigma(m)\sigma(n)}.
$$
We define a multiplicative function $H$ as follows : for any prime $p$ and 
positive integer $k$,
$$
H(p^k) = 1 - \frac{\dsi(p)^2 p}{h(p)\sigma(p)^2}.
$$
We note that $H(p) \ge 1/5$ for all primes $p \ge 5$ and $H(2), H(3) \ge 1/6$.
This implies that
\begin{align*}
M & ~=
\log (D_1X) \sum_{\gamma \ge 1} \frac{H(\gamma)}{h(\gamma)}
\sum_{m,n \ge 1}
r(m\gamma)r(n\gamma)\psi_0(m\gamma)\overline{\psi_0(n\gamma)}
~\frac{\dsi(m)\dsi(n)}{h(m)h(n)}
\frac{\sqrt{mn}}{\sigma(m)\sigma(n)}.
\end{align*}
Now using multiplicativity of the functions, we have
\begin{align*}
M & ~=
\log (D_1X) \sum_{\gamma \ge 1} \frac{H(\gamma)}{h(\gamma)}
r(\gamma)^2|\psi_0(\gamma)|^2  \!\!
\sum_{m \ge 1 \atop{(m, \gamma) = 1}}
r(m)\psi_0(m)\frac{\dsi(m)}{h(m)}
\frac{\sqrt{m}}{\sigma(m)} \!\!
\sum_{n \ge 1 \atop{(n, \gamma) = 1}}r(n)\overline{\psi_0(n)}
\frac{\dsi(n)}{h(n)}
\frac{\sqrt{n}}{\sigma(n)}\\
& ~=
\log (D_1X) \sum_{\gamma \ge 1} \frac{H(\gamma)}{h(\gamma)}r(\gamma)^2|\psi_0(\gamma)|^2
\prod_{\ell \nmid \gamma}
\Bigg\vert 1 + 
r(\ell)\psi_0(\ell)\frac{\dsi(\ell)}{h(\ell)}
\frac{\sqrt{\ell}}{\sigma(\ell)}\Bigg\vert^2.
\end{align*}
Define a multiplicative function $g(n)$ as follows: for any prime $p$ and 
positive integer $k$,
$$
g(p^k) = \Bigg(1 + 
r(p^k)\psi_0(p)\frac{\dsi(p)}{h(p)}
\frac{\sqrt{p}}{\sigma(p)}\Bigg).
$$
Note that $|g(p)| > \frac{1}{20}$ for all primes $p$ since $h(p) \ge 1 + \frac{(p-1)^2}{p^2(p+1)}$
 for all primes $p$. Therefore, we get the main term as 
\begin{align}\label{EstM}
M & ~=
(\log X + \log D_1) \prod_{\ell}\Bigg(
\vert g(\ell)\vert^2 + \frac{H(\ell)}{h(\ell)}r(\ell)^2|\psi_0(\ell)|^2\Bigg).
\end{align}
We now estimate $E_1$ as follows. Note that
\begin{align}\label{FirstE1}
E_1 & ~=
\sum_{\gamma \ge 1} \frac{1}{h(\gamma)}
\sum_{\beta | \gamma}\frac{\mu(\beta)\dsi(\beta)^2\beta}{h(\beta)\sigma(\beta)^2}
\sum_{m,n \ge 1 }
r(m\gamma)r(n\gamma)\psi_0(m\gamma)\overline{\psi_0(n\gamma)}
~\frac{\dsi(m)\dsi(n)}{h(m)h(n)}
\frac{\sqrt{mn}}{\sigma(m)\sigma(n)} \nonumber\\
&~ \qquad \Bigg( \sum_{p|m}\Big(\frac{D_2(p)}{p}-1 \Big)\log p
+ \sum_{p|n}\Big(\frac{D_2(p)}{p}-1 \Big)\log p \Bigg)
\nonumber \\
& ~=
\sum_{\gamma \ge 1} \frac{H(\gamma)}{h(\gamma)}
r(\gamma)^2|\psi_0(\gamma)|^2
\Bigg\{\sum_{m \ge 1 \atop{(m, \gamma)=1}}
r(m)\psi_0(m)
\frac{\dsi(m)}{h(m)}\frac{\sqrt{m}}{\sigma(m)}
\sum_{n \ge 1 \atop{(n, \gamma)=1}}
r(n)\overline{\psi_0(n)}
\frac{\dsi(n)}{h(n)}\frac{\sqrt{n}}{\sigma(n)}\nonumber \\
&~ \qquad \Bigg( \sum_{p|m}\Big(\frac{D_2(p)}{p}-1 \Big)\log p
+ \sum_{p|n}\Big(\frac{D_2(p)}{p}-1 \Big)\log p \Bigg)
\Bigg\}.
\end{align}
Interchanging the sums over $m$, we get
\begin{align*}
& ~~
\sum_{m \ge 1 \atop{(m, \gamma)=1}}
r(m)\psi_0(m)
\frac{\dsi(m)}{h(m)}\frac{\sqrt{m}}{\sigma(m)}
\sum_{p|m}\Big(\frac{D_2(p)}{p}-1 \Big)\log p \\
= &  ~
\sum_{p \nmid \gamma}\Big(\frac{D_2(p)}{p}-1 \Big)
r(p)\psi_0(p)\frac{\dsi(p)}{h(p)}
\frac{\sqrt{p}\log p}{\sigma(p)}
\sum_{m \ge 1 \atop{(m, p\gamma)=1}}r(m)\psi_0(m)
\frac{\dsi(m)}{h(m)}
\frac{\sqrt{m}}{\sigma(m)}\\
 = & ~
\Big(\prod_{\ell \nmid \gamma}g(\ell)\Big)
\sum_{p \nmid \gamma}\Big(\frac{D_2(p)}{p}-1 \Big)
r(p)\frac{\dsi(p)}{h(p)}
\frac{\sqrt{p}\log p}{\sigma(p)}\frac{\psi_0(p)}{g(p)}.
\end{align*}
Similarly, interchanging the sums over $n$, we have
\begin{align*}
& ~~
\sum_{n \ge 1 \atop{(n, \gamma)=1}}
r(n)\overline{\psi_0(n)}
\frac{\dsi(n)}{h(n)}\frac{\sqrt{n}}{\sigma(n)}
\sum_{p|n}\Big(\frac{D_2(p)}{p}-1 \Big)\log p \\
 = & ~
\Big(\prod_{\ell \nmid \gamma}\overline{g(\ell)}\Big)
\sum_{p \nmid \gamma}\Big(\frac{D_2(p)}{p}-1 \Big)
r(p)\frac{\dsi(p)}{h(p)}
\frac{\sqrt{p}\log p}{\sigma(p)}
\frac{\overline{\psi_0(p)}}{ \overline{g(p)}}.
\end{align*}
Thus we can write \eqref{FirstE1} as
\begin{align*}
E_1 & ~=
\sum_{\gamma \ge 1} \frac{H(\gamma)}{h(\gamma)}
r(\gamma)^2|\psi_0(\gamma)|^2
\Bigg\{
\Big(\prod_{\ell \nmid \gamma}|g(\ell)|^2\Big)
\sum_{p \nmid \gamma}\Big(\frac{D_2(p)}{p}-1 \Big)
r(p)\frac{\dsi(p)}{h(p)}
\frac{\sqrt{p}\log p}{\sigma(p)}
\Bigg(\frac{\psi_0(p)}{g(p)} + \frac{\overline{\psi_0(p)}}{\overline{g(p)}}\Bigg)\Bigg\}\\
& ~=
\Big(\prod_{\ell}|g(\ell)|^2\Big)\sum_{\gamma \ge 1} \frac{H(\gamma)}{h(\gamma)}
r(\gamma)^2|\psi_0(\gamma)|^2
\Big(\prod_{\ell |\gamma}|g(\ell)|^{-2}\Big)\\
& \qquad \Bigg\{\sum_{p \nmid \gamma}\Big(\frac{D_2(p)}{p}-1 \Big)
r(p)\frac{\dsi(p)}{h(p)}
\frac{\sqrt{p}\log p}{\sigma(p)|g(p)|^2}
\Bigg( 
d_{\psi_0}(p) + 2r(p)|\psi_0(p)|^2\frac{\dsi(p)}{h(p)}\frac{\sqrt{p}}{\sigma(p)}
\Bigg)\Bigg\}.
\end{align*}
Taking absolute values, we get
\begin{align}\label{EstE_1}
E_1 & \ll
\Big(\prod_{\ell}|g(\ell)|^2\Big)\sum_{\gamma \ge 1} \frac{H(\gamma)}{h(\gamma)}
r(\gamma)^2|\psi_0(\gamma)|^2
\Big(\prod_{\ell |\gamma}|g(\ell)|^{-2}\Big)
\Bigg\{\sum_{p \nmid \gamma}
\frac{r(p)\log p}{\sqrt{p}}
\Bigg( 
1 + \frac{r(p)}{\sqrt{p}}
\Bigg)\Bigg\}\nonumber \\
& \ll
\Big(\prod_{\ell}|g(\ell)|^2\Big)\sum_{\gamma \ge 1} \frac{H(\gamma)}{h(\gamma)}
r(\gamma)^2|\psi_0(\gamma)|^2
\Big(\prod_{\ell |\gamma}|g(\ell)|^{-2}\Big)
\Bigg\{\sum_{L^2 \le p \le L^4}
\Bigg(\frac{L}{p} + \frac{L^2}{p^2\log p}
\Bigg)\Bigg\}\nonumber\\
& \ll
 \sqrt{\log X \log\log X}\prod_{\ell}\Bigg(
\vert g(\ell)\vert^2 + \frac{H(\ell)}{h(\ell)}r(\ell)^2|\psi_0(\ell)|^2\Bigg).
\end{align}
Next we estimate $E_2$. We have
\begin{align}\label{FirstE_2}
E_2 & ~=
\sum_{\gamma \ge 1} \frac{1}{h(\gamma)}
\sum_{\beta | \gamma}\frac{\mu(\beta)\dsi(\beta)^2\beta}{h(\beta)\sigma(\beta)^2}
\sum_{m,n \ge 1 }
r(m\gamma)r(n\gamma)\psi_0(m\gamma)\overline{\psi_0(n\gamma)}
~\frac{\dsi(m)\dsi(n)}{h(m)h(n)}
\frac{\sqrt{mn}}{\sigma(m)\sigma(n)}\nonumber\\
&~ \qquad \Bigg(2\sum_{p|\beta}\Big(\frac{D_2(p)}{p}-1 \Big)\log p
+ \sum_{p|\frac{\gamma}{\beta}}\frac{D_3(p)}{p}\log p \Bigg)\nonumber \\
& ~= 
\sum_{\gamma \ge 1} 
\frac{r(\gamma)^2}{h(\gamma)}|\psi_0(\gamma)|^2
\Big(\prod_{\ell \nmid \gamma}|g(\ell)|^2\Big)
\nonumber\\
&~ \qquad \Bigg\{\sum_{\beta | \gamma}\frac{\mu(\beta)
\dsi(\beta)^2\beta}{h(\beta)\sigma(\beta)^2}
\Big( 2\sum_{p|\beta}\Big(\frac{D_2(p)}{p}-1 \Big)\log p
+ \sum_{p|\frac{\gamma}{\beta}}\frac{D_3(p)}{p}\log p \Big)\Bigg\}.
\end{align}
Interchanging the sums in the inner term, we have
\begin{align*}
&
\sum_{\beta | \gamma}
\frac{\mu(\beta)\dsi(\beta)^2\beta}{h(\beta)\sigma(\beta)^2}
\Big( 2\sum_{p|\beta}\Big(\frac{D_2(p)}{p}-1 \Big)\log p
+ \sum_{p|\frac{\gamma}{\beta}}\frac{D_3(p)}{p}\log p \Big)\\
&~ =
2\sum_{p|\gamma}\Big(\frac{D_2(p)}{p}-1 \Big)\log p
\sum_{\beta | \frac{\gamma}{p}}\mu(p\beta)\frac{\dsi(p\beta)^2p\beta}{h(p\beta)\sigma(p\beta)^2}
+ 
\sum_{p|\gamma}\frac{D_3(p)}{p}\log p 
\sum_{\beta | \frac{\gamma}{p}}\mu(\beta)\frac{\dsi(\beta)^2\beta}{h(\beta)\sigma(\beta)^2}\\
& ~ =
2\sum_{p|\gamma}\Big(1 - \frac{D_2(p)}{p} \Big)
\frac{\dsi(p)^2}{h(p)}\frac{p\log p}{\sigma(p)^2}
\sum_{\beta | \frac{\gamma}{p}}\mu(\beta)\frac{\dsi(\beta)^2\beta}{h(\beta)\sigma(\beta)^2}
+ 
\sum_{p|\gamma}\frac{D_3(p)}{p}\log p 
\sum_{\beta | \frac{\gamma}{p}}\mu(\beta)\frac{\dsi(\beta)^2\beta}{h(\beta)\sigma(\beta)^2}\\
& = H(\gamma) \Bigg\{
2\sum_{p|\gamma}\Big(1 - \frac{D_2(p)}{p} \Big)
\frac{\dsi(p)^2}{h(p)H(p)}\frac{p\log p}{\sigma(p)^2}
+ 
\sum_{p|\gamma}\frac{D_3(p)}{H(p)}\frac{\log p}{p} \Bigg\}.
\end{align*}
Putting this in \eqref{FirstE_2}, we get
\begin{align}\label{EstE_2}
E_2 
& \ll
\sum_{\gamma \ge 1} 
r(\gamma)^2\frac{H(\gamma)}{h(\gamma)}|\psi_0(\gamma)|^2
\Big(\prod_{\ell \nmid \gamma}|g(\ell)|^2\Big)
\Bigg\{
\sum_{p|\gamma}
\frac{\log p}{p} \Bigg\}\nonumber \\
& \ll\log\log X
\prod_{\ell}\Bigg(
\vert g(\ell)\vert^2 + \frac{H(\ell)}{h(\ell)}r(\ell)^2|\psi_0(\ell)|^2\Bigg).
\end{align}
Using \eqref{EstS_1}, \eqref{EstM}, \eqref{EstE_1} and \eqref{EstE_2}, we get
\begin{equation}\label{finalS_1}
S_1 =
\Big(\log X + O(\sqrt{\log X \log\log X})\Big)
\prod_{\ell}\Big(
\vert g(\ell)\vert^2 + \frac{H(\ell)}{h(\ell)}r(\ell)^2|\psi_0(\ell)|^2\Big) .
\end{equation}
Following the arguments as before, we see that
\begin{eqnarray}\label{S1Psi}
S_{1, \psi} 
&=&
\sum_{m,n \ge 1}  r(m) r(n)\psi_0(m)\overline{\psi_0(n)}
\frac{d\big(\frac{mn}{(m,n)^2}\big)}{h_{\psi}(mn)}
\frac{ \psi\big(\frac{mn}{(m,n)^2}\big) \sqrt{\frac{mn}{(m,n)^2}}}{ \prod_{p|\frac{mn}{(m,n)^2}}( p + \psi(p)^2)}
\\
& = &
\prod_{p}\Big( 1 + \frac{2r(p)d_{\psi_0}(p)\psi(p)\sqrt{p}}{h_{\psi}(p)(p+ \psi(p)^2)} 
+ \frac{4r(p)^2|\psi_0(p)|^2\psi(p)^2 p}{h_{\psi}(p)^2(p+ \psi(p)^2)^2} 
+ r(p)^2 |\psi_0(p)|^2 \frac{H_{\psi}(p)}{h_{\psi}(p)}\Big). \nonumber
\end{eqnarray}
Here $H_{\psi}$ is a multiplicative function defined by
$$
H_{\psi}(p^k) 
= 1- \frac{4 \psi(p)^2 p}{h_{\psi}(p )(p+ \psi(p)^2)^2}
$$
for primes $p$ and $k \ge 1$.
Note that $|H_{\psi}(p^k)| \ge 2/9$ for all primes $p \ge 7$. By replacing $\psi$ by $\bar{\psi}$,
we get analogous expression for $S_{1, \bar{\psi}}$.
Now we estimate contribution from the terms with $\max\{ m, n \} >N$. Let $\alpha=\frac{1}{8\log L}$.  
Following the arguments as before, we get
\begin{align*}
S_2 & =~ 
\sum_{\max\{m,n\} \ge N}r(m)r(n)\psi_0(m)\overline{\psi_0(n)}
\frac{\dsi\big(\frac{mn}{(m,n)^2}\big)}{h(mn)}
\frac{\sqrt{\frac{mn}{(m,n)^2}}}{\sigma\big(\frac{mn}{(m,n)^2}\big)}\\
& \qquad \Big(\log(D_1X) + \sum_{p|\frac{mn}{(m,n)^2}}\Big(\frac{D_2(p)}{p}-1 \Big)\log p
+\sum_{p|(m,n)}\frac{D_3(p)}{p}\log p \Big)\\
& ~ =
\sum_{t \ge 1}\sum_{\max\{m,n\} \ge N/t \atop{(m,n)=1}}
r(mt)r(nt)\psi_0(mt)\overline{\psi_0(nt)}
~\frac{\dsi(m)\dsi(n)}{h(m)h(n)h(t)}
\frac{\sqrt{mn}}{\sigma(m)\sigma(n)}\\
&~ \qquad \Big( \log(D_1X) + \sum_{p|m}\Big(\frac{D_2(p)}{p}-1 \Big)\log p
+ \sum_{p|n}\Big(\frac{D_2(p)}{p}-1 \Big)\log p
+ \sum_{p|t}\frac{D_3(p)}{p}\log p \Big).
\end{align*}
Using $\sum_{\beta | b}\mu(\beta)=0$ unless $b = 1$ and multiplicativity 
of the functions in the sum, we get
\begin{align}\label{EstS_2}
S_2 & ~=
\sum_{t \ge 1}\frac{1}{h(t)}
\sum_{\beta \ge 1}\mu(\beta)\frac{\dsi(\beta)^2\beta}{h(\beta)^2\sigma(\beta)^2}
\Bigg\{\sum_{\max\{m,n\} \ge  \frac{N}{\beta t}}
r(m\beta t)r(n\beta t)\psi_0(m\beta t)\overline{\psi_0 (n\beta t)} \nonumber\\ 
 & ~ \qquad 
 \frac{\dsi(m)\dsi(n)}{h(m)h(n)}
\frac{\sqrt{mn}}{\sigma(m)\sigma(n)}
 \Bigg(\log(D_1X) + \sum_{p|m}\Big(\frac{D_2(p)}{p}-1 \Big)\log p
+ \sum_{p|n}\Big(\frac{D_2(p)}{p}-1 \Big)\log p \nonumber\\
&~ \qquad \qquad \qquad
+ 2\sum_{p|\beta}\Big(\frac{D_2(p)}{p}-1 \Big)\log p
+ \sum_{p|t}\frac{D_3(p)}{p}\log p \Bigg)\Bigg\}\nonumber\\
& ~=
\sum_{\gamma \ge 1}\frac{1}{h(\gamma)}
\sum_{\beta | \gamma}\mu(\beta)\frac{\dsi(\beta)^2\beta}{h(\beta)\sigma(\beta)^2}
\Bigg\{\sum_{\max\{m,n\} \ge  \frac{N}{ \gamma}}
r(m \gamma)r(n \gamma)\psi_0(m \gamma)\overline{\psi_0(n \gamma)} \nonumber\\ 
 & ~ \qquad 
 \frac{\dsi(m)\dsi(n)}{h(m)h(n)}
\frac{\sqrt{mn}}{\sigma(m)\sigma(n)}
 \Bigg(\log(D_1X) + \sum_{p|m}\Big(\frac{D_2(p)}{p}-1 \Big)\log p
+ \sum_{p|n}\Big(\frac{D_2(p)}{p}-1 \Big)\log p\nonumber\\
&~ \qquad \qquad \qquad
+ 2\sum_{p|\beta}\Big(\frac{D_2(p)}{p}-1 \Big)\log p
+ \sum_{p|\frac{\gamma}{\beta}}\frac{D_3(p)}{p}\log p \Bigg)\Bigg\}\nonumber\\
& = E_3 ~+~ E_4 ~+~ E_5, \text{ say}.
\end{align}
Here $E_3$ consists of the terms involving $\log(D_1X)$,  $E_4$ consists of the terms 
involving $m, n$ and $E_5$ consists of the terms 
involving $\beta, \gamma/\beta$ in the innermost sum. We first estimate $E_3$
as follows; 
\begin{align*}
E_3 
& =
\log(D_1X) 
\sum_{\gamma \ge 1} \frac{H(\gamma)}{h(\gamma)}
\sum_{\max\{m,n\} \ge  \frac{N}{\gamma}}
r(m\gamma)r(n\gamma)\psi_0(m\gamma)\overline{\psi_0(n\gamma)}
~\frac{\dsi(m)\dsi(n)}{h(m)h(n)}
\frac{\sqrt{mn}}{\sigma(m)\sigma(n)}\\
& \ll 
\log X
\sum_{\gamma \ge 1} \frac{H(\gamma)}{h(\gamma)}
\sum_{m,n \ge 1}\frac{(mn\gamma)^{\alpha}}{N^{\alpha}}
r(m\gamma)r(n\gamma)
~\frac{|\dsi(m)|}{h(m)}\frac{|\dsi(n)|}{h(n)}
\frac{\sqrt{mn}}{\sigma(m)\sigma(n)} \\
& \ll 
N^{-\alpha}\log X
\sum_{\gamma \ge 1} \frac{H(\gamma)}{h(\gamma)}r(\gamma)^2\gamma^{\alpha}
\sum_{m \ge 1 \atop{(m, \gamma)=1}}
r(m)\frac{|\dsi(m)|}{h(m)}
~\frac{m^{1/2 + \alpha}}{\sigma(m)}
\sum_{n \ge 1 \atop{(n, \gamma)=1}}
r(n) \frac{|\dsi(n)|}{h(n)}
\frac{n^{1/2 + \alpha}}{\sigma(n)}. 
\end{align*}
Define a multiplicative function $G(n)$ as follows : 
for any prime $p$ and positive integer $k$, 
$$
G(p^k) = 1 + r(p^k)\frac{|\dsi(p)|p^{1/2 + \alpha}}{h(p)\sigma(p)}.
$$
This implies that
\begin{equation}\label{EstE_3}
E_3 \ll N^{-\alpha}\log X \prod_{\ell}\Bigg(G(\ell)^2 + 
 \frac{H(\ell)}{h(\ell)}r(\ell)^2\ell^{\alpha}\Bigg).
\end{equation}
Similarly, we estimate $E_4$ as follows;
\begin{align*}
E_4 
& \ll
\sum_{\gamma \ge 1} \frac{H(\gamma)}{h(\gamma)}
\sum_{\max\{m,n\} \ge \frac{ N}{\gamma}}
r(m\gamma)r(n\gamma)
~\frac{|\dsi(m)||\dsi(n)|}{h(m)h(n)}
\frac{\sqrt{mn}}{\sigma(m)\sigma(n)}
\Big(\sum_{p|m}\log p
+ \sum_{p|n}\log p\Big)\\
& \ll
\sum_{\gamma \ge 1} \frac{H(\gamma)}{h(\gamma)}
\sum_{m,n \ge 1}\frac{(mn\gamma)^{\alpha}}{N^{\alpha}}
r(m\gamma)r(n\gamma)
~\frac{|\dsi(m)|}{h(m)}\frac{|\dsi(n)|}{h(n)}
\frac{\sqrt{mn}}{\sigma(m)\sigma(n)} \Big(\sum_{p|m}\log p
+ \sum_{p|n}\log p\Big)\\
& \ll N^{-\alpha}
\sum_{\gamma \ge 1} \frac{H(\gamma)}{h(\gamma)}
r(\gamma)^2\gamma^{\alpha}
\sum_{m \ge 1 \atop{(m, \gamma)=1}}
r(m)\frac{|\dsi(m)|}{h(m)}\frac{m^{1/2 +\alpha}}{\sigma(m)}
\Bigg(\sum_{n \ge 1 \atop{(n, \gamma)=1}}r(n)
\frac{|\dsi(n)|}{h(n)}
\frac{n^{1/2 +\alpha}}{\sigma(n)}\sum_{p|n}\log p\Bigg).
\end{align*}
Interchanging the sums in the innermost sum, we get
\begin{align}\label{EstE_4}
E_4 & \ll 
N^{-\alpha}
\sum_{\gamma \ge 1} \frac{H(\gamma)}{h(\gamma)}
r(\gamma)^2\gamma^{\alpha}
\prod_{\ell \nmid \gamma}G(\ell)^2
\Bigg(\sum_{p \nmid \gamma}r(p)
\frac{|\dsi(p)|}{h(p)}
\frac{p^{1/2 +\alpha}\log p}{\sigma(p)G(p)}\Bigg)\nonumber\\
& \ll
N^{-\alpha}
\sum_{\gamma \ge 1} \frac{H(\gamma)}{h(\gamma)}
r(\gamma)^2\gamma^{\alpha}
\prod_{\ell \nmid \gamma}G(\ell)^2
\Bigg(\sum_{L^2 \le p \le L^4}r(p)
\frac{|\dsi(p)|}{h(p)}
\frac{p^{1/2 +\alpha}\log p}{\sigma(p)G(p)}\Bigg)\nonumber\\
& \ll N^{-\alpha}\sqrt{\log X \log\log X} \prod_{\ell}\Bigg(G(\ell)^2 + 
 \frac{H(\ell)}{h(\ell)}r(\ell)^2\ell^{\alpha}\Bigg).
\end{align}
Finally, we estimate $E_5$. Arguing as in the estimate of $E_2$, we get
\begin{align*}
E_5 &  ~=
\sum_{\gamma \ge 1} \frac{H(\gamma)}{h(\gamma)}
\sum_{\max\{m,n\} \ge  \frac{N}{\gamma}}
r(m\gamma)r(n\gamma)\psi_0(m\gamma)\overline{\psi_0(n\gamma)}
~\frac{\dsi(m)\dsi(n)}{h(m)h(n)}
\frac{\sqrt{mn}}{\sigma(m)\sigma(n)}\\
&\qquad  \Bigg(
2\sum_{p|\gamma}\Big(1 - \frac{D_2(p)}{p} \Big)
\frac{\dsi(p)^2}{h(p)H(p)}\frac{p\log p}{\sigma(p)^2}
+ 
\sum_{p|\gamma}\frac{D_3(p)}{H(p)}\frac{\log p}{p} \Bigg).
\end{align*}
Taking absolute values, we get
\begin{align}\label{EstE_5}
E_5 & \ll
\sum_{\gamma \ge 1} \frac{H(\gamma)}{h(\gamma)}
\sum_{\max\{m,n\} \ge  \frac{N}{\gamma}}
r(m\gamma)r(n\gamma)
~\frac{|\dsi(m)||\dsi(n)|}{h(m)h(n)}
\frac{\sqrt{mn}}{\sigma(m)\sigma(n)}\Bigg(
\sum_{p|\gamma}\frac{\log p}{p} \Bigg)\nonumber\\
& ~\ll
\sum_{\gamma \ge 1} \frac{H(\gamma)}{h(\gamma)}
\sum_{m,n \ge 1}\frac{(mn\gamma)^{\alpha}}{N^{\alpha}}
r(m\gamma)r(n\gamma)
~\frac{|\dsi(m)|}{h(m)}\frac{|\dsi(n)|}{h(n)}
\frac{\sqrt{mn}}{\sigma(m)\sigma(n)} \Bigg(
\sum_{p|\gamma}\frac{\log p}{p} \Bigg)\nonumber\\
& ~\ll
N^{-\alpha}\sum_{\gamma \ge 1} 
\frac{H(\gamma)}{h(\gamma)}r(\gamma)^2\gamma^{\alpha}
\Big(\sum_{m \ge 1 \atop{(m,\gamma) = 1}}
\frac{r(m)|\dsi(m)|}{h(m)}
\frac{m^{1/2 + \alpha}}{\sigma(m)}\Big)^2 \Bigg(
\sum_{L^2 \le p \le L^4}\frac{\log p}{p} \Bigg)\nonumber\\
& ~\ll
N^{-\alpha}\log\log X \prod_{\ell} \Big(G(\ell)^2 
+ \frac{H(\ell)}{h(\ell)}r(\ell)^2\ell^{\alpha}
\Big).
\end{align}
Using \eqref{EstS_2}, \eqref{EstE_3}, \eqref{EstE_4} and \eqref{EstE_5}, we get
\begin{equation}\label{finalS_2}
S_2 
~\ll~
 N^{-\alpha}\log X
 \prod_{\ell}\Bigg(G(\ell)^2 + 
 \frac{H(\ell)}{h(\ell)}r(\ell)^2\ell^{\alpha}\Bigg)
~\ll~
\exp\Big(-\frac{\log X}{400\log \log X}\Big)
\end{equation}
for sufficiently large $X$. Using similar arguments, we can show for sufficiently large $X$ that 
\begin{eqnarray}\label{S2Psi}
S_{2, \psi}
&=&
\sum_{ \max \{m,n \} \ge N}  r(m) r(n)\psi_0(m)\overline{\psi_0(n)}
\frac{d\big(\frac{mn}{(m,n)^2}\big)}{h_{\psi}(mn)}
\frac{ \psi\big(\frac{mn}{(m,n)^2}\big) \sqrt{\frac{mn}{(m,n)^2}}}{ \prod_{p|\frac{mn}{(m,n)^2}}( p + \psi(p)^2)} \nonumber \\
&\ll&
\exp\Big(-\frac{\log X}{400\log \log X}\Big).
\end{eqnarray}
Replacing $\psi$ with $\bar{\psi}$, we get similar bounds for $S_{2, \bar{\psi}}$. 
Using \eqref{Modified}, \eqref{finalS_1}, \eqref{S1Psi}, \eqref{finalS_2} and \eqref{S2Psi}, we get
\begin{eqnarray}\label{part1}
&&
\!\!\!\!\!\!\!\!\!
\sideset{}{'}\sum_{d \atop d \equiv a \!\!\! \pmod{q}} |L(1/2,\poc)|^2 |R(d)|^2
\Phi\Big(\frac{d}{X}\Big) \\
&&
\!\!\!\!\!\!\!\!\!
=
\frac{A_{\psi}}{\varphi(q)}\hat{\Phi}(0)X\Big(\log X + O(\sqrt{\log X \log\log X})\Big)
\prod_{p}\Big(
\vert g(p)\vert^2 + \frac{H(p)}{h(p)}r(p)^2|\psi_0(p)|^2\Big)
~+~ \frac{4\hat{\Phi}(0)X}{\varphi(q)} \nonumber \\
&& 
 \Re\Bigg( \overline{\psi(a)}\Big(\frac{\pm q}{a}\Big)
C_{\psi}\prod_{p}\Big(1 + \frac{2r(p)d_{\psi_0}(p)\psi(p)\sqrt{p}}{h_{\psi}(p)(p+ \psi(p)^2)} 
+ \frac{4r(p)^2|\psi_0(p)|^2\psi(p)^2 p}{h_{\psi}(p)^2(p+ \psi(p)^2)^2} +
 r(p)^2 |\psi_0(p)|^2 \frac{H_{\psi}(p)}{h_{\psi}(p)}\Big) \Bigg). \nonumber
\end{eqnarray}
Using the fact that $\frac{1}{h(p)}=1+O(\frac{1}{p})$ and $H(p)=1+O(\frac{1}{p})$,
we see that 
\begin{align*}
&~~
\vert g(p)\vert^2 + \frac{H(p)}{h(p)}r(p)^2|\psi_0(p)|^2\\
= &~ 
1 + r(p)\frac{d_{\psi_0}(p)\dsi(p)}{h(p)}\frac{\sqrt{p}}{\sigma(p)} 
+  r(p)^2|\psi_0(p)|^2\frac{\dsi(p)^2}{h(p)^2}\frac{p}{\sigma(p)^2} 
+ \frac{H(p)}{h(p)} 
r(p)^2|\psi_0(p)|^2\\
= &~ 
1 + r(p)^2|\psi_0(p)|^2+ 
\frac{r(p)}{\sqrt{p}} d_{\psi_0}(p)\dsi(p)
+ O\Big(\frac{r(p)^2}{p}\Big)\\
= &~  
\Big(1 + r(p)^2|\psi_0(p)|^2\Big)
\Big(1 + \frac{r(p)}{\sqrt{p}}d_{\psi_0}(p)\dsi(p)
+ O\Big(\frac{r(p)^3}{\sqrt{p}} \Big)\Big)\\ 
= &~
\Big(1 + r(p)^2|\psi_0(p)|^2\Big)
\exp\Big(\frac{r(p)}{\sqrt{p} } d_{\psi_0}(p)\dsi(p)
  + O\Big(\frac{r(p)^3}{\sqrt{p}}\Big)\Big).
\end{align*}
Therefore the product in the first term of the right hand side 
of \eqref{part1} becomes 
\begin{eqnarray}\label{coprimeE}
&&
\prod_p \Big(1 + r(p)^2|\psi_0(p)|^2\Big)
\exp\Bigg(\sum_p \frac{r(p)}{\sqrt{p}}  d_{\psi_0}(p)\dsi(p)
  + O\Big(\sum_p\frac{r(p)^3}{\sqrt{p}}\Big)\Bigg) \nonumber \\
&=&
\prod_p (1+r(p)^2|\psi_0(p)|^2)
\exp\Big( \sum_p \frac{r(p)}{\sqrt{p}}d_{\psi_0}(p)\dsi(p) 
+ O\Big(\frac{L}{\log^3L}\Big) \Big).
\end{eqnarray}
Similarly, using the fact that 
$\frac{1}{h_{\psi}(p)}=1+O(\frac{1}{p})$ and $H_{\psi}(p)=1+O(\frac{1}{p})$,
we have
\begin{eqnarray*}
&&
1 + \frac{2r(p)d_{\psi_0}(p)\psi(p)\sqrt{p}}{h_{\psi}(p)(p+ \psi(p)^2)} 
+ \frac{4r(p)^2|\psi_0(p)|^2\psi(p)^2 p}{h_{\psi}(p)^2(p+ \psi(p)^2)^2} +
r(p)^2 |\psi_0(p)|^2 \frac{H_{\psi}(p)}{h_{\psi}(p)}\\
& = &
\Big( 1 + r(p)^2 |\psi_0(p)|^2 \Big)
\exp\Big( \frac{2r(p)d_{\psi_0}(p)\psi(p)}{\sqrt{p}}  
+ O\Big( \frac{r(p)^3}{\sqrt{p}} \Big)\Big)
\end{eqnarray*}
and hence the product in the second term on the right hand side of \eqref{part1} becomes
\begin{equation}\label{prodPsi}
\prod_{p}\Big( 1 + r(p)^2 |\psi_0(p)|^2 \Big)
\exp\Big(\sum_p \frac{2r(p)}{\sqrt{p}}d_{\psi_0}(p)\psi(p)  
+ O\Big( \frac{L}{\log^3L} \Big)\Big).
\end{equation}
For any co-prime residue class $a$ modulo $2q$, it follows from 
\eqref{part1}, \eqref{coprimeE} and \eqref{prodPsi} that
\begin{eqnarray}\label{classpart1}
&&
\sideset{}{'}\sum_{d \atop{d \equiv a \!\!\! \!\!\!\pmod{2q}}} |L(1/2,\poc)|^2 |R(d)|^2
\Phi\Big(\frac{d}{X}\Big)
\nonumber \\
& = &
\frac{A_{\psi}\hat{\Phi}(0)}{\varphi(q)}X\Big(\log X + O(\sqrt{\log X \log\log X})\Big)
\prod_p (1+r(p)^2|\psi_0(p)|^2) \nonumber \\
&& 
\exp\Big( \sum_p \frac{r(p)}{\sqrt{p}}d_{\psi_0}(p)\dsi(p) 
+ O\Big(\frac{L}{\log^3L}\Big) \Big)
~+~\frac{4\hat{\Phi}(0)X}{\varphi(q)}\prod_{p}\Big( 1 + r(p)^2 |\psi_0(p)|^2 \Big)   \nonumber \\
&& 
\Re\Bigg( \overline{\psi(a)}\Big(\frac{\pm q}{a}\Big)
C_{\psi}
\exp\Big(\sum_p \frac{2r(p)}{\sqrt{p}}d_{\psi_0}(p)\psi(p)  
+ O\Big( \frac{L}{\log^3L} \Big)\Big) \Bigg).
\end{eqnarray}
Taking sum over all co-prime residue classes modulo $2q$
and using orthogonality of characters, we note that
the second term in \eqref{classpart1} does not contribute. 
Choose $\Phi$ to be a smooth function supported on $[1,2]$ with 
$0 \leq \Phi(t) \leq 1$ for all $t$ and $\Phi(t)=1$ on $[1.1 , 1.9]$. 
Then we have
\begin{eqnarray*}
&& \DD |L(1/2,\psi_0 \otimes \cchi)|^2|R(d)|^2 \\
& \geq &
\sideset{}{'}\sum_{d \atop (d,2q)=1} |L(1/2, \psi_0 \otimes \cchi)|^2|R(d)|^2
\Phi\Big(\frac{d}{X}\Big)\\
& \geq &
\frac45 A_{\psi_0}X\Big(\log X + O(\sqrt{\log X \log\log X}~)\Big) 
\prod_p (1+ r(p)^2|\psi_0(p)|^2)\\
&& 
\qquad
\exp\Big(\sum_{p} \frac{r(p)}{\sqrt{p}}d_{\psi_0}(p)^2 
~+~ O\Big(\frac{L}{\log^3L}\Big) \Big). 
\end{eqnarray*}
Next, we choose $\Phi$ to be a smooth function 
supported on $[0.9, 2.1]$ with $0 \leq \Phi(t) \leq 1$ for all $t$ and
$\Phi(t)=1$ on $[1 , 2]$. Then, using \eqref{classpart1} and noting that
the cardinality of $\mathcal{A}$ is at most $\varphi(q)/3$, we have
\begin{eqnarray*}
&&
\sideset{}{'}\sum_{\substack{X < d \le 2X \\ 
{d \equiv a \!\!\!\!\!\pmod{2q} \text{ for } a \in \mathcal{A}}}}
|L(1/2,\psi_0 \otimes \cchi)|^2|R(d)|^2 \\
& \leq &
\sideset{}{'}\sum_{\substack{d \equiv a \!\!\!\!\!\pmod{2q} 
\text{ for } a \in \mathcal{A}}  }
|L(1/2, \psi_0 \otimes \cchi )|^2|R(d)|^2
\Phi\Big(\frac{d}{X}\Big)\\
& \leq &
\frac25 A_{\psi_0}X\Big(\log X + O(\sqrt{\log X \log\log X}~)\Big) 
\\
&& 
\prod_p (1+ r(p)^2|\psi_0(p)|^2)
\exp\Big(\sum_{p} \frac{r(p)}{\sqrt{p}}d_{\psi_0}(p)^2
~+~ O\Big(\frac{L}{\log^3L}\Big) \Big) \nonumber \\
&+&
\frac{8}{5}X\prod_{p}\Big( 1 + r(p)^2 |\psi_0(p)|^2 \Big) 
\left|\Re\Bigg( \overline{\psi_0(a)}\Big(\frac{\pm q}{a}\Big)
C_{\psi_0}
\exp\Big(\sum_p \frac{2r(p)}{\sqrt{p}}d_{\psi_0}(p)\psi_0(p) 
+ O\Big( \frac{L}{\log^3L} \Big)\Big) \Bigg)\right| .
\end{eqnarray*}
Therefore 
\begin{eqnarray*}
&&
\sideset{}{'}\sum_{\substack{X < d \le 2X \\ 
{(d, 2q)=1 \atop d \not\equiv a \!\!\!\!\!\pmod{2q} \text{ for any } a \in \mathcal{A} }}}
|L(1/2,\psi_0 \otimes \cchi)|^2 |R(d)|^2\\
&=&
\DD |L(1/2,\psi_0 \otimes \cchi)|^2|R(d)|^2
~-~
\sideset{}{'}\sum_{\substack{X < d \le 2X \\ 
{d \equiv a \!\!\!\!\!\pmod{2q} \text{ for } a \in \mathcal{A}}  }}
|L(1/2,\psi_0 \otimes \cchi)|^2|R(d)|^2\\
& \geq &
\frac25 A_{\psi_0}X\Big(\log X + O(\sqrt{\log X \log\log X}~)\Big)\\
&&
\prod_p (1+r(p)^2|\psi_0(p)|^2) 
\exp\Big(\sum_{ p } \frac{r(p)}{\sqrt{p}}d_{\psi_0}(p)^2 
~+~ O\Big(\frac{L}{\log^3L}\Big) \Big) \\
&-&
\frac{8}{5}X\prod_{p}\Big( 1 + r(p)^2 |\psi_0(p)|^2 \Big) 
\left|\Re\Bigg( \overline{\psi_0(a)}\Big(\frac{\pm q}{a}\Big)
C_{\psi_0}
\exp\Big(\sum_p \frac{2r(p)}{\sqrt{p}}d_{\psi_0}(p)\psi_0(p) 
+ O\Big( \frac{L}{\log^3L} \Big)\Big) \Bigg)\right|. 
\end{eqnarray*}
Using \eqref{psisum1}, we note that
$$
\sum_{ p} \frac{r(p)}{\sqrt{p}}d_{\psi_0}(p)^2 
~=~
\Big(\frac{1}{2}+o(1)\Big)\frac{L}{\log L}.
$$
Similarly
\begin{eqnarray*}
\sum_{ p } \frac{r(p)}{\sqrt{p}}(2 + 2\psi_0(p)^2) 
&=&
\Big(\frac{1}{2}+o(1)\Big)\frac{L}{\log L}.
\end{eqnarray*}
This implies that
\begin{align*}
&
\sideset{}{'}\sum_{\substack{X < d \le 2X,  \\ 
{(d, 2q) =1 \atop d \not\equiv a \!\!\!\!\!\pmod{2q} \text{ for any } a \in \mathcal{A}}  }}
|L(1/2,\psi_0 \otimes \cchi)|^2 |R(d)|^2 \\
 &\gg
X\log X \prod_p (1+r(p)^2)\exp\Big(\Big(\frac{1}{2}+o(1)\Big)\frac{L}{\log L}\Big)\\
&\gg
X\log X \prod_p (1+r(p)^2)\exp\Bigg(\Big(\frac{1}{8\sqrt{24}}+o(1)\Big)
\sqrt{\frac{\log X}{\log\log X}}\Bigg).
\end{align*}
When $\psi \neq \psi_0, \bar{\psi}_0$, choose a smooth function $\Phi$ 
supported on $[1/2,5/2]$ with $0 \leq \Phi(t) \leq 1$ for all $t$ and
$\Phi(t)=1$ on $[1 , 2]$.
Then from \eqref{classpart1}, we have
\begin{eqnarray*}
&& \DD |L(1/2,\poc)|^2|R(d)|^2 \\
& \leq &
\sideset{}{'}\sum_{d \atop (d,2q)=1} |L(1/2,\poc)|^2|R(d)|^2
\Phi\Big(\frac{d}{X}\Big)\\
&\leq &
2 A_{\psi}X\Big(\log X + O(\sqrt{\log X \log\log X}~)\Big) 
 \prod_p (1+r(p)^2|\psi_0(p)|^2)\\
&& \qquad
\exp\Big(\sum_{p } \frac{r(p)}{\sqrt{p}}d_{\psi_0}(p)\dsi(p) 
~+~ O\Big(\frac{L}{\log^3L}\Big) \Big).
\end{eqnarray*}
When $\psi \neq \psi_0, \bar{\psi}_0$, applying \eqref{psisum2}, we get 
$$
\sum_p \frac{r(p)}{\sqrt{p}}d_{\psi_0}(p)\dsi(p) 
=
o\Big(\frac{L}{\log L}\Big).
$$
This implies that whenever $\psi \neq \psi_0, \bar{\psi}_0$, we have
\begin{align*}
\DD |L(1/2,\poc)|^2|R(d)|^2 
~\ll~
X\log X \prod_p (1+ r(p)^2)\exp\Bigg(o\Bigg(\sqrt{\frac{\log X}{\log\log X}}\Bigg)\Bigg).
\end{align*}
This completes the proof of \propref{M2}.
\end{proof} 

\smallskip

\section{Proof of \thmref{mainthm}}
Since $F$ is non-zero, $c_{\psi_0} \neq 0$ for some $\psi_0$. Without
loss of generality, we can assume that $\psi_0$ is even and $c_{\psi_0}=1$.  
Then for any square-free positive integer $d$ with $(d, 2q)=1$, we have
$$
L(1/2, ~F \otimes \cchi)
~= ~~
\sideset{}{^*}\sum_{~\psi\!\!\!\!\!\pmod{q}} c_{\psi} 
L(1/2,~\psi \otimes \cchi),
$$
where the sum varies over even, non-quadratic Dirichlet 
characters $\psi$ modulo $q$ such that $\psi^2$ is primitive modulo $q$. 
 Applying Cauchy-Schwarz inequality, we have
\begin{eqnarray*}
&&
\Big|L(1/2, ~\psi_0 \otimes \cchi) + c_{\bar{\psi}_0}L(1/2, ~\bar{\psi}_0 \otimes \cchi) \Big|^2 \\
&\leq&
2~\Big| \sideset{}{^*}\sum_{~\psi\!\!\!\!\!\pmod{q}} c_{\psi} 
L(1/2,~\psi \otimes \cchi)\Big|^2  + 
 2~ \Big| \sideset{}{^*}\sum_{~\psi\!\!\!\!\!\pmod{q} \atop{\psi \neq \psi_0, \bar{\psi}_0}}
 c_{\psi} L(1/2, ~\psi \otimes \cchi)\Big|^2 \\ 
 &\leq&
 2~\Big| \sideset{}{^*}\sum_{~\psi\!\!\!\!\!\pmod{q}} c_{\psi} 
L(1/2,~\psi \otimes \cchi)\Big|^2  +
2\Big (\sideset{}{^*}\sum_{~\psi\!\!\!\!\!\pmod{q} \atop{\psi \neq \psi_0, \bar{\psi}_0}}
 |c_{\psi}|^2 \Big)
 \Big( \sideset{}{^*}\sum_{~\psi\!\!\!\!\!\pmod{q} \atop{\psi \neq \psi_0, \bar{\psi}_0}}
  | L(1/2, ~\psi \otimes \cchi)|^2 \Big).
\end{eqnarray*}
Choosing $A(\psi_0) = 2\displaystyle\sideset{}{^*}\sum_{\psi \!\!\pmod{q} 
\atop{\psi \neq \psi_0, \bar{\psi}_0}}
 |c_{\psi}|^2 $, we get 
\begin{eqnarray}\label{F8d}
2~|L(1/2, ~F \otimes \cchi)|^2
&\geq&
\Big|L(1/2, ~\psi_0 \otimes \cchi) + c_{\bar{\psi}_0}L(1/2, ~\bar{\psi}_0 
\otimes \cchi) \Big|^2 \nonumber \\
&&- \phantom{m}
A(\psi_0)~\sideset{}{^*}\sum_{~\psi\!\!\!\!\!\pmod{q} \atop{\psi \neq \psi_0, \bar{\psi}_0}}
 |L(1/2, ~\psi \otimes \cchi)|^2.
\end{eqnarray}
Note that by \eqref{fun}, we have
$$
L(1/2,\psi_0 \otimes \cchi) ~=~ \epsilon_0(d) \bar {L(1/2, \psi_0 \otimes \cchi)},
$$
where 
$$
\epsilon_0 (d)= \psi_0(8d)\Big(\frac{8d}{q}\Big)
\frac{\tau(\psi_0)}{\sqrt{q}}.
$$ 
Thus
\begin{eqnarray}\label{M1}
\Big|L(1/2, ~\psi_0 \otimes \cchi) 
+ c_{\bar{\psi}_0} L(1/2, ~\bar{\psi}_0 \otimes \cchi) \Big|^2 
~~=~~
\Big|L(1/2, ~\psi_0 \otimes \cchi) \Big|^2    
\Big| 1 + c_{\bar{\psi}_0}\overline{\epsilon_0(d)}  \Big|^2.
\end{eqnarray}
If $1 + c_{\bar{\psi}_0}\overline{\epsilon_0(d)} \ne 0$, then there is a positive
absolute constant $B(\bar {\psi}_0)$ depending on $\bar {\psi}_0, q$
such that 
$$
\Big| 1 + c_{\bar{\psi}_0}\overline{\epsilon_0(d)}  \Big|^2 
~>~  B(\bar {\psi}_0) ~>~ 0.
$$
Indeed, values of $\bar{\psi_0}(\cdot) \Big(\frac{\cdot}{q}\Big)$ form a discrete 
subset of the unit circle. When $1 + c_{\bar{\psi}_0}\overline{\epsilon_0(d)} = 0$, 
then there are at most $\varphi(q)/3$ co-prime residue classes $b \!\!\pmod q$
such that $d \equiv b \!\! \pmod q$. This follows from the fact that 
$\bar{\psi_0}(\cdot)\Big(\frac{\cdot}{q}\Big)$ is a non-quadratic character
and hence it takes at least three distinct values so that 
the kernel has at most $\varphi(q)/3$ elements.
We denote these co-prime residue classes modulo $q$ by 
$b_1 \cdots, b_{J(\bar{\psi}_0, q)}$
and their unique lifts (since $q$ is odd) to co-prime residue classes modulo $2q$ by 
$a_1 \cdots, a_{J(\bar{\psi}_0, q)}$.
Let
$$
T(\psi_0, q) 
~=~ 
\{ d >0   ~:~  d \text{ square-free}, (d, 2q)=1, ~d \not\equiv a_i \!\!\!\!\!\pmod{2q} 
\text{ for } 1 \le i \le J(\bar{\psi}_0, q) \} 
$$
and $\mathcal{S}$ be the set of $d \in T(\psi_0, q)$ such that $X< d \leq 2X$ and 
\begin{equation}\label{DefS}
 B(\bar {\psi}_0) |L(1/2,\psi_0 \otimes \cchi)|^2
~>~
A(\psi_0)\sideset{}{^*}\sum_{\psi\!\!\!\!\! \pmod{q} \atop{\psi \neq \psi_0, \bar{\psi}_0}}
|L(1/2, \psi \otimes \cchi)| ^2
~+~ 
2 \exp\left(\frac{1}{40}\sqrt{\frac{\log X}{\log\log X}}\right).
\end{equation}
Thus, whenever $d \in \mathcal{S}$, it follows from \eqref{F8d}, \eqref{M1}, 
the above discussions and \eqref{DefS} that
\begin{equation}
|L(1/2, ~F \otimes \cchi)| ~>~ \exp\left(\frac{1}{80}\sqrt{\frac{\log X}{\log\log X}}\right).
\end{equation}
Note that
\begin{eqnarray}\label{fthm}
&& \sum_{\substack{X < d \le 2X \\ d \in T(\psi_0, q)}} |L(1/2,\psi_0 \otimes \cchi)|^2
|R(d)|^2 \nonumber \\
&\leq& 
\sum_{d \in \mathcal{S}}
|L(1/2,\psi_0 \otimes \cchi)|^2|R(d)|^2   
~+~  
\frac{2}{B(\bar{\psi}_0)} \exp\Bigg(\frac{1}{40}\sqrt{\frac{\log X}{\log\log X}}\Bigg)
\sum_{\substack{X < d \le 2X \\ d \in T(\psi_0, q)}} |R(d)|^2
  \nonumber \\
&& + \phantom{m} 
\frac{A(\psi_0)}{B(\bar{\psi}_0)} ~\sum_{\substack{X < d \le 2X \\ d \in T(\psi_0, q)}} 
~~\sideset{}{^*}\sum_{\substack{\psi \!\!\!\!\!\pmod{q} \\ {\psi \neq \psi_0, \bar{\psi}_0}}}
|L(1/2, \psi \otimes \cchi)| ^2 |R(d)|^2.    
\end{eqnarray} 
Applying \propref{M2}, we see that 
\begin{align}\label{LeL}
\DDD|L(1/2,\psi_0 \otimes \cchi)|^2 |R(d)|^2 
\phantom{mmmmmmmmmmmmmmm} \nonumber
&&\\
\gg ~~
X\log X \prod_p(1+ r(p)^2)\exp\Bigg(\Big(\frac{1}{8\sqrt{24}}+o(1)\Big)
\sqrt{\frac{\log X}{\log\log X}}\Bigg).&&
\end{align}
Using \propref{brd}, we have
\begin{equation}\label{endbd1}
\sum_{\substack{X < d \le 2X \\ d \in T(\psi_0, q)}} |R(d)|^2 
~\ll~ 
X  \prod_p (1 + r(p)^2).
\end{equation}
Further for any $\psi \ne \psi_0, \bar{\psi}_0$, applying \propref{M2}, we get
\begin{equation}\label{endbd2}
\DDD |L(1/2,\poc)|^2|R(d)|^2 
~\ll~
X\log X \prod_p (1 + r(p)^2)
\exp\Bigg(o\Bigg(\sqrt{\frac{\log X}{\log\log X}}\Bigg)\Bigg).
\end{equation}
Substituting \eqref{LeL}, \eqref{endbd1} and \eqref{endbd2} in \eqref{fthm},
we obtain
\begin{align}\label{lower}
X\log X \prod_p (1+r(p)^2)\exp\Big(\Big(\frac{1}{8\sqrt{24}}+o(1)\Big)
\sqrt{\frac{\log X}{\log\log X}}\Big) 
\ll~
 \sum_{d \in \mathcal{S}} |L(1/2,\psi_0 \otimes \cchi)|^2|R(d)|^2.
\end{align}
Finally applying H\"older's inequality, \propref{brd}
and \propref{moment}, we see that
\begin{align}\label{upper}
\sum_{d \in \mathcal{S}} |L(1/2,\psi_0 \otimes \cchi)|^2|R(d)|^2
&\ll~
\Big( \DD |L(1/2,\psi_0 \otimes \cchi)|^4 \Big)^{1/2} 
\Big( \DD |R(d)|^6 \Big)^{1/3}
|\mathcal{S}|^{1/6} \nonumber\\
&\ll~
X^{1/2+\epsilon} X^{1/3+\epsilon} |\mathcal{S}|^{1/6} 
~\ll~ 
X^{5/6 +\epsilon}|\mathcal{S}|^{1/6}.
\end{align}
Using \eqref{lower} and \eqref{upper},
we conclude that
$|\mathcal{S}| \gg X^{1- \epsilon}.$
This completes the proof of \thmref{mainthm}.
\qed

\bigskip
\noindent
{\bf Acknowledgments.}~ We are grateful to K. Soundararajan for
sharing his PhD thesis with us and patiently answering our queries 
about his paper. We are also
grateful to the referee for a very careful reading of the article and
pertinent suggestions. We would like
to thank the Institute of Mathematical Sciences (a constituent of
HBNI), Chennai Mathematical Institute, 
Max-Planck-Institut f\"ur Mathematik, Alfr\'ed R\'enyi Institute of Mathematics and Sorbonne Universit\'e 
for providing excellent working environment 
where parts of the work were done. The first author would also like to thank 
ReLaX program and YuBi60 conference organizers 
for partial financial support during the final stages of the project. The second author is currently supported by Momentum MSCA programme. The Momentum MSCA Programme is co-funded by the European Commission through the HORIZON-MSCA-2023-COFUND programme and the Secretariat of the Hungarian Academy of Sciences (MTA), Grant Agreement No. 101179854.

\end{document}